\documentclass[a4paper,11pt,leqno,twoside]{article}
\usepackage{amsmath,amsthm,amssymb}
\usepackage[dvips]{graphicx,color,psfrag}
\usepackage{rotating}
\usepackage{multirow}

\numberwithin{equation}{section}

\newcommand{\bR}{\mathbb{R}}
\newcommand{\bZ}{\mathbb{Z}}

\newcommand{\cS}{\mathcal{S}}
\newcommand{\cL}{\mathcal{L}}

\newcommand{\m}{\mu}

\newcommand{\RB}{\mathrm{RB}}

\newcommand{\Gauss}[1]{\lfloor{#1}\rfloor}

\newcommand{\red}[1]{{\color{red}{#1}}}
\newcommand{\blue}[1]{{\color{blue}{#1}}}

\newcommand{\boldtitle}[1]{\title{\bfseries #1}}

\newcommand{\I}{\mathrm{I}}
\newcommand{\II}{\mathrm{I\hspace{-.1em}I}}
\newcommand{\III}{\mathrm{I\hspace{-.1em}I\hspace{-.1em}I}}

\newenvironment{MSC}{%
\smallbreak
\noindent \textbf{2010\ Mathematics Subject Classification\,:}}

\newenvironment{keywords}{%
\noindent\textbf{Key words and phrases\,:}\itshape}

\theoremstyle{theorem}

\theoremstyle{definition}
\newtheorem*{multiproclaim}{\variable@name}

\theoremstyle{plain}
\newtheorem{thm}{Theorem}[section]
\newtheorem{prop}[thm]{Proposition}
\newtheorem{lem}[thm]{Lemma}

\newtheorem{conj}[thm]{Conjecture}

\theoremstyle{definition}

\newtheorem{example}[thm]{Example}
\newtheorem{remark}[thm]{Remark}

\boldtitle{\bf Ramanujan circulant graphs and the conjecture of Hardy-Littlewood and Bateman-Horn} 
\author{
 Miki HIRANO\thanks{Partially supported by Grant-in-Aid for Scientific Research (C) No. 24540022.},
 Kohei KATATA
 and
 Yoshinori YAMASAKI
\thanks{Partially supported by Grant-in-Aid for Young Scientists (B) No. 24740018.}
}
\date{\today}

\begin{document}
%===========================================================================
%===========================================================================

\setlength{\baselineskip}{15pt}
\maketitle 

\begin{abstract}
 In this paper, we determine the bound of the valency of the odd circulant graphs
 which guarantees to be Ramanujan for each fixed number of vertices.
 In almost of the cases, the bound coincides with the trivial bound,
 which comes from the trivial estimate of the largest non-trivial eigenvalue of the circulant graph.
 As exceptional cases, the bound in fact exceeds the trivial one by two.
 We then prove that such exceptionals occur only in the cases 
 where the number of vertices has at most two prime factors and is represented by a quadratic polynomial in a finite
 family and, moreover, under the conjecture of Hardy-Littlewood and Bateman-Horn, 
 exist infinitely many.
\begin{MSC}
 {\it Primary}
 11M41,
% Other Dirichlet series and zeta functions For local and global ground fields,  
% 05C50;
% Graphs and linear algebra (matrices, eigenvalues, etc.) 
 {\it Secondary}
 05C25,
% Graphs and abstract algebra (groups, rings, fields, etc.) 
 05C75,
% Structural characterization of families of graphs
 11N32.
% Primes represented by polynomials
\end{MSC} 
\end{abstract}

%\tableofcontents

\section{Introduction}

 Let $X$ be a regular graph with standard assumptions, that is, finite, undirected, connected and simple. 
 Spectral analysis on $X$ is an important topic in several interest of mathematics,
 such as combinatorics, group theory, differential geometry, and number theory. 
 Especially, the topics around Ramanujan and expander graphs are focused;
 these are related each other and have common interest in
 the second eigenvalue (or the spectral gap) of the adjacency operator on $X$
({\it cf}. \cite{HooryLinialWigderson2006,Lubotzky2012}). 
%\cite{Lubotzky2012}. 

 The notion of Ramanujan graph was defined in \cite{{LubotzkyPhillipsSarnak1988}}:
 The graph $X$ is called {\it Ramanujan} if its largest non-trivial eigenvalue (in the sense of absolute value)
 is not greater than the Ramanujan bound $2\sqrt{k-1}$, where $k$ is the valency (or degree) of $X$. 
 In view of the theory of zeta functions, the Ramanujan property means that
 the associated Ihara zeta function satisfies the ``Riemann hypothesis''.  
 Here the Ihara zeta functions are regarded as a graph analogue
 of the Selberg zeta functions on locally symmetric spaces.  
 Similarly to the case of the usual prime number (or geodesic) theorem,
 one has a good estimate for the number of the prime cycles in $X$ if it is Ramanujan ({\it cf}. \cite{Terras2011}). 
 Therefore, for a given graph, we want to examine whether it is Ramanujan or not in easy way. 

 As one can be seen from the estimation of the isoperimetric constant,
 the Ramanujan graphs are very much connected in some sense. 
 The complete graph $K_m$ with $m$-vertices
 which is the densest graph with the eigenvalues $\{m-1,-1,\cdots,-1\}$ is in fact Ramanujan.  
 Also, some neighbors of $K_m$ are expected to be Ramanujan ({\it cf}. \cite{AlonRoichman1994}). 
 Now, we want to estimate the precise boundary of the number of removable edges
 from the complete graph preserving the Ramanujan property. 
% In other words, we ask when the Ramanujan property of $X$ is guaranteed only by its valency. 

 We formulate our problem in the general setting. 
 Let $\mathcal{G}$ be the set of all (isomorphic classes of) graphs with standard assumptions and
 $\mathcal{G}_{m,k}$ the subset of $\mathcal{G}$ consisting of the graphs with $m$-vertices and $k$-valency. 
 Similarly, let $\mathcal{R}$ and $\mathcal{R}_{m,k}$ be the set of all Ramanujan graphs
 in $\mathcal{G}$ and $\mathcal{G}_{m,k}$, respectively. 
 For a given subset $\mathcal{X}$ of $\mathcal{G}$,
 we put $\mathcal{X}_{m,k}=\mathcal{X}\cap\mathcal{G}_{m,k}$ and decide the set 
\[
 \Gamma_m=\bigl\{k\in\mathcal{V}_m\,\bigl|\,\mathcal{X}_{m,k}\subset\mathcal{R}_{m,k}\bigr\}
\]
 for each $m$, where $\mathcal{V}_m=\{k\,|\,\mathcal{X}_{m,k}\ne\emptyset\}$. 
 If $K_m$ can be realized in $\mathcal{X}$,
 we have $m-1\in\Gamma_m$ because $\mathcal{G}_{m,m-1}=\mathcal{R}_{m,m-1}=\{K_m\}$. 

 In this paper,
 we take as $\mathcal{X}$ the easiest family of the Cayley graphs, that is,
 the set of odd circulant graphs (i.e., the Cayley graphs of cyclic groups $\mathbb{Z}_{m}$ of odd order $m$),
 and try to decide
\[
 \hat{l}_m
=\max\bigl\{l\in \mathbb{N}\,\bigl|\,m-l\in\mathcal{V}_m,\,[m-l,m-1]\cap\mathcal{V}_m\subset\Gamma_m\bigr\},
\qquad m\in2\mathbb{N}+1.
\]
 Here $\hat{l}_m$ means the maximum number of edge-removal preserving the Ramanujan property
 from the odd complete graph $K_m=\mathrm{Cay}(\mathbb{Z}_m,\mathbb{Z}_m\setminus\{0\})\in\mathcal{X}$.
 Our main result is the following theorem which says our problem associates a classical problem in analytic number theory. 

\begin{thm}
\label{thm:MainResult}
 Let $\mathcal{X}$ be the family of circulant graphs of odd order.
 Then, for $m\ge 15$, we have 
\[
 \hat{l}_m=l_{0,m}+\varepsilon_m,
\]
 where $l_{0,m}=2 \Gauss{\sqrt{m}-\frac{3}{2}}+1$ and $\varepsilon_m\in\{0,2\}$.
 Here $\Gauss{x}$ denotes the largest integer not exceeding $x$.
 Moreover, the case $\varepsilon_m=2$ occurs only
 if $m$ is  represented by one of the quadratic polynomials $k^2+5k+c$
 for some $c\in\{\pm 1,\pm 3,\pm 5\}$
 and is either a prime or a product of two distinct primes $p,q$ with $p<q<4p$.
\end{thm}

 We notice that $l_{0,m}$ comes from a trivial estimate of the largest non-trivial eigenvalue.
 We also remark that the above condition is not sufficient; if $q$ is very close to $4p$,
 then one can in fact observe that $\varepsilon_{m}=0$ even if $m=pq$ can be represented by one of the above quadratic polynomials ({\it cf}. \S4).
 Let us call $m$ {\it ordinary} (resp. {\it exceptional}) if $\varepsilon_m=0$ (resp. $\varepsilon_m=2$).
 Our result suggests that
 the existence of infinitely many exceptionals for our $\mathcal{X}$ is related to the
 well-known conjecture
 of Hardy-Littlewood \cite{HardyLittlewood1923} and Bateman-Horn \cite{BatemanHorn1962} on primes represented by polynomials,
 or to the Iwaniec's important result  \cite{Iwaniec1978} (see also the recent result \cite{LemkeOliver2012})
 on the almost-primes represented by a quadratic polynomial. 

 We also consider the case of odd abelian family in the final section. 
 For the case of even circulant family,
 we need a slightly different formulation coming from the symmetricity of generating sets for the groups. 
 We treat this case in another paper \cite{Katata2013}. 
 Moreover, we will discuss the case of the simplest non-abelian family,
 that is, the dihedral family, in \cite{HiranoKatataYamasaki}.

\section{Preliminaries}

\subsection{Cayley graphs and their eigenvalues}
\label{sec:CayleyGraph}

 Let $X$ be a $k$-regular graph with $m$-vertices which is finite, undirected, connected, and simple. 
 The {\it adjacency matrix} $A_X$ {\it of} $X$ is the symmetric matrix of size $m$
 whose entry is $1$ if the corresponding pair of vertices are connected by an edge and $0$ otherwise. 
 We call the eigenvalues of $A_X$ the {\it eigenvalues of} $X$. 
 The set $\Lambda(X)$ of all eigenvalues of $X$ is given as 
\[
 \Lambda(X)=\bigl\{\lambda_i\,\bigl|\,k=\lambda_0>\lambda_1\geq\cdots\geq\lambda_{m-1}\geq -k\bigr\}.
\]
 Remark that $-k\in\Lambda(X)$ if and only if $X$ is bipartite ({\it cf.} \cite{DavidoffSarnakValette2003}). 
 Let $\m(X)$ be the largest non-trivial eigenvalue of $X$ in the sense of absolute value, that is, 
\[
 \m(X)=\max\bigl\{|\lambda|\,\bigl|\,\lambda\in \Lambda{(X)}, \ |\lambda|\ne k\bigr\}.
\]
 Then, $X$ is called {\it Ramanujan} if the inequality $\mu(X)\le 2\sqrt{k-1}$ holds. 
 Here the constant $2\sqrt{k-1}$ in the right hand side of this inequality is often
 called the {\it Ramanujan bound for} $X$ and is denoted by $\RB(X)$. 

 Let $G$ be a finite group with the identity element $e$
 and $S$ a {\it Cayley subset of} $G$, that is, a symmetric set of generators for $G$ satisfying $e\not\in S$. 
 Then, the {\it Cayley graph} $X(G,S)$ is the $|S|$-regular graph with vertex set $G$ and the edge set
 $\{(x,y)\in G\times G\,|\,x^{-1}y\in S\}$, which is undirected, connected, and simple. 
 The adjacency matrix of $X(G,S)$ is described in terms of the right regular representation of $G$ ({\it cf.} \cite{Terras1999}). 
 In particular, if $G$ is a finite abelian group,
 then we have 
\[
 \Lambda\bigl(X(G,S)\bigr)
=\left\{\left.\lambda_{\chi}=\sum_{s\in S}\chi(s)\,\right|\,\chi\in\hat{G}\right\}.
\]
 Here $\hat{G}$ is the dual group of $G$. 

\subsection{A problem for Ramanujan circulants} 

 Fix $m$ a positive integer and let $\bZ_m=\bZ/m\bZ$ be the cyclic group of order $m$. 
 Moreover, put $\cS$ the set of all Cayley subsets of $\bZ_m$. 
 We call a Cayley graph $X(S)=X(\bZ_m,S)$ with $S\in\cS$ a {\it circulant graph} of order $m$. 
 Since the dual group of $\bZ_m$ consists of the characters
 $\chi_j(a)=e^{\frac{2\pi iaj}{m}}$ ($0\le j\le m-1$),
 the set of all eigenvalues of $X(S)$ is given by
\[
 \Lambda(X(S))=\bigl\{\mu_j(S)\,\bigl|\,0\le j\le m-1\bigr\},
\]
 where $\mu_0(S)=|S|$ and 
\begin{equation}
\label{def:eigenvaluesa}
 \m_j(S)
=\sum_{a\in S}e^{\frac{2\pi iaj}{m}}
=-\sum_{b\in \bZ_m\setminus S}e^{\frac{2\pi ibj}{m}}, \qquad 1\le j\le m-1. 
\end{equation}

 Besides the valency $|S|$ of a circulant graph $X(S)$,
 we call $l(S)=|\bZ_m\setminus S|=m-\vert S\vert$ the {\it covalency} of $X(S)$. 
 Now we divide the set $\cS$ of Cayley subsets of $\bZ_m$ by the covalency as 
\[
 \cS=\bigsqcup_{l\in\cL}\cS_l, \quad \cS_l=\{S\in\cS\,|\,l(S)=l\}.
\]
 Here $\cL=\{l(S)\,|\,S\in\cS\}$ is the set of values of covalency. 
 For example, we have $\bZ_m\setminus\{0\}\in\cS_1$ and $\{\pm1\}\in\cS_{m-2}$
 for which the attached graphs are the complete graph $K_m=X(\bZ_m\setminus\{0\})$
 and the cycle graph $C_m=X(\{\pm1\})$ of order $m$, respectively. 
 These Cayley subsets give the non-trivial eigenvalues 
\[
 \m_j\bigl(\bZ_m\setminus\{0\}\bigr)=-1,\quad
 \m_j\bigl(\{\pm 1\}\bigr)=2\cos\frac{2\pi j}{m}, \qquad 1\leq j\le m-1.
\]
 Moreover, if we put 
\begin{equation}
\label{def:Sl}
 S^{(l)}=\bZ_{m}\setminus\Bigl\{0,\pm 1,\pm 2,\ldots,\pm\frac{l-1}{2}\Bigr\}
\end{equation}
 for an odd integer $l$ with $1\le l\leq m-2$,
 then one sees that $S^{(l)}$ is an element of $\cS_{l}$ with the non-trivial eigenvalues 
\begin{equation}
\label{for:eigenSl}
 \m_{j}\bigl(S^{(l)}\bigr)
=-\sum^{\frac{l-1}{2}}_{b=-\frac{l-1}{2}}e^{\frac{2\pi ibj}{m}}=-\frac{\sin\frac{\pi jl}{m}}{\sin\frac{\pi j}{m}},
 \qquad 1\leq j\le m-1.
\end{equation}
 The Cayley subset $S^{(l)}$ often appears in our discussion.

 From the definition, the circulant graph $X(S)$ is Ramanujan
 if and only if $\m(S)\le \RB(S)$ where $\mu(S)=\mu(X(S))$ and $\RB(S)=\RB(X(S))$. 
 Observe that the Ramanujan bound $\RB(S)=2\sqrt{m-l-1}$ is depend only on the covalency $l=l(S)$ of $S\in\cS_l$. 
 Moreover, we remark that $\cS_1=\{\bZ_m\setminus\{0\}\}$ and $X(\bZ_m\setminus\{0\})$ is a Ramanujan circulant
 because $\m(\bZ_m\setminus\{0\})=|-1|\le 2\sqrt{m-2}=\RB(\bZ_m\setminus\{0\})$.
 These observations naturally lead us to evaluate the bound
\begin{align*}
 \hat{l}
&=\max
\left\{
 l\in \cL \,\left|\,\text{$X(S)$ is Ramanujan for all $S\in\bigsqcup_{k\in\cL \atop 1\le k\le l}\cS_k$}
\right.\right\},
\end{align*}
 which means the maximal number of edge-removal from the complete graph $K_m$ preserving the Ramanujan property.  
 In particular, $\hat{l}=m-2$ is equivalent to say that $X(S)$ is Ramanujan for all $S\in\cS$.

 In this paper, we treat only the case of odd $m$. 
 Then, each $S\in\cS$ has even number of elements because of symmetry,
 and hence $\cL=\{1,3,\cdots,m-2\}$ consists of odd integers. 
 Moreover, $-|S|$ does not appear in $\Lambda(X(S))$ because $X(S)$ has odd vertices and thus is not bipartite. 
 (It is known that $X(S)$ is bipartite if and only if $m$ is even and all the elements of $S$ are odd. See, e.g., \cite{Heuberger2003}.)
 Therefore, we have $\m(S)=\max\{|\m_j(S)|\,|\,1\le j\le m-1\}$.

\section{Initial results}

 The following lemma says that, on the determination of $\hat{l}$, we may assume that $m\ge 15$.

\begin{lem}
\label{prop:finite}
 $\hat{l}=m-2$ if and only if $3\le m\le 13$.
\end{lem}
\begin{proof}
 Remark that
 the cycle graph $C_{m}=X(\{\pm 1\})$ is Ramanujan,
% because $\mu(\{\pm 1\})=|2\cos{\frac{2\pi}{m}}|<2=\RB(\{\pm 1\})$.
 whence $X(S)$ is whenever $|S|=2$. % because $X(S)$ is isomorphic to $C_{m}$.
 Therefore, to prove the ``only if'' part, 
 it suffices to show that  
 there exists $S\in\cS_{m-4}$ such that $X(S)$ is not Ramanujan for $m\ge 15$.
 Actually, let $S=\{\pm\frac{m-1}{2},\pm\frac{m-3}{2}\}\in \cS_{m-4}$.
 % which is an element of $\cS_{m-4}$. % because $(\frac{m-1}{2},m)=1$.
 From \eqref{def:eigenvaluesa}, we have 
 $|\m_1(S)|=4\cos\frac{\pi}{m}\cos\frac{2\pi}{m}$, which is monotonic increasing. % with respect to $m$. 
 Hence, for $m\ge 15$, 
\[
 \m(S)
\ge |\m_1(S)|
=4\cos\frac{\pi}{m}\cos\frac{2\pi}{m}
\ge 4\cos\frac{\pi}{15}\cos\frac{2\pi}{15}
=3.57\ldots
%>3.46\ldots
>2\sqrt{3}
=\RB(S).
\]
 The converse is direct.
\end{proof}

\subsection{Trivial bound}

 We first show that
 there exists a lower bound of $\hat{l}$.

\begin{lem}
\label{lem:trivial}
 We have $\hat{l}\ge l_0$, where
\[
 l_0=2\Gauss{\sqrt{m}-\frac{3}{2}}+1.
\]
 Here $\Gauss{x}$ denotes the largest integer not exceeding $x$.
\end{lem}
\begin{proof}
 From \eqref{def:eigenvaluesa}, for any $S\in\cS_{l}$ with $1\le l<\frac{m}{2}$, 
 we have % a trivial estimate 
 $|\m_j(S)|\le\min\{\sum_{a\in S}1,\sum_{b\in \mathbb{Z}_m\setminus S}1\}\le\min\{|S|,l(S)\}=l(S)=l$ for all $1\le j\le
 m-1$
 and hence $\mu(S)\le l$.
 Therefore, if $l\le\RB(S)=2\sqrt{m-l-1}$, equivalently $l\le 2(\sqrt{m}-1)$,
 then $X(S)$ is Ramanujan.
 Now, the claim follows because  
 $l_0$ coincides with the maximum odd integer satisfying $l<\frac{m}{2}$ and $l\le 2(\sqrt{m}-1)$.
% $l_0=\max\{l\in \cL\,|\,l\le \frac{m}{2}, \ l\le 2(\sqrt{m}-1) \}$
% (notice that all the elements in $\cL$ are odd).
\end{proof} 

 We call $l_0$ the {\it trivial bound of} $\hat{l}$.
 Note that, since $x-1<\Gauss{x}\le x$, we have  
\begin{equation}
\label{for:boundtrivialbound}
 2(\sqrt{m}-2)<l_0\le 2(\sqrt{m}-1)
\end{equation}
 and hence $l_0\sim 2\sqrt{m}$ as $m\to+\infty$.
 See the table below for the explicit value of $l_0$ and $\hat{l}$ for small $m$
 (remark that $\hat{l}=m-2$ for $3\le m\le 13$ from Lemma~\ref{prop:finite}).
 
\begin{table}[htmb]
\begin{center}
{\renewcommand\arraystretch{1.3}
\begin{tabular}{c||c|c|c|c|c|c||c|c|c|c|c|c|c|c}
       $m$ & 3 & 5 & 7 & 9 & 11 & 13 & 15 & 17 & 19 & 21 & 23 & 25 & 27 & 29 \\
\hline
\hline
     $l_0$ &  &  &  & &  &  &  5 &  5 &  5 &  7 &  7 &  7 &  7 &  7 \\
\hline
 $\hat{l}$ & 1 & 3 & 5 & 7 &  9 & 11 &  7 &  7 &  7 &  7 &  9 &  9 &  7 &  9 
\end{tabular}
}

\ \\[10pt]

{\renewcommand\arraystretch{1.3}
\begin{tabular}{c||c|c|c|c|c|c|c|c|c|c|c|c|c}
       $m$ & 31 & 33 & 35 & 37 & 39 & 41 & 43 & 45 & 47 & 49 & 51 & 53 & 55  \\
\hline
\hline
     $l_0$ &  9 &  9 &  9 &  9 &  9 &  9 & 11 & 11 & 11 & 11 & 11 & 11 & 11 \\ 
\hline
 $\hat{l}$ &  9 &  9 & 11 & 11 &  9 & 11 & 11 & 11 & 13 & 13 & 11 & 13 & 13
\end{tabular}
}

\caption{$l_0$ and $\hat{l}$ for small $m$.}
\end{center}
\end{table}

\subsection{Beyond the trivial bound}

 As you find from Table~1,
 we can indeed prove the following theorem.
% which asserts that the difference between $\hat{l}$ and $l_0$ is at most two.

\begin{thm}
\label{thm:oddmaxbound}
 There exists $\varepsilon\in\{0,2\}$ such that $\hat{l}=l_0+\varepsilon$ for $m\ge 15$.
\end{thm}

 To prove the theorem,
 it is sufficient to show that there exists $S\in\cS_{l_0+4}$ such that $X(S)$ is not Ramanujan.
 Actually, for large $m$, we claim that $X(S^{(l_0+4)})$ is not Ramanujan where $S^{(l)}$ is defined in \eqref{def:Sl}. 
 More strongly, we show the following 

\begin{lem}
\label{prop:nonRamanujanex}
 $X(S^{(l_0+2h)})$ is not Ramanujan if $m\ge 39$ and $2\le h\le \Gauss{\frac{1}{4}(\sqrt{m}-2)^2}$.
\end{lem}
\begin{proof}
 Assume that $\Gauss{\frac{1}{4}(\sqrt{m}-2)^2}\ge 2$.
 Using \eqref{for:boundtrivialbound},
 we have $2(\sqrt{m}-(2-h))<l_0+2h<\frac{m}{2}$.
% since $\m(S^{(l_0+2h)})\ge |\m^{1}(S^{(l_0+2h)})|$.
 Then the expression \eqref{for:eigenSl} together with the inequality above
 leads us to the evaluation 
\begin{align}
\label{for:estimate}
 |\m_1(S^{(l_0+2h)})|-\RB(S^{(l_0+2h)})
&>\frac{\sin{\frac{2\pi(\sqrt{m}-(2-h))}{m}}}{\sin{\frac{\pi}{m}}}-2\bigl(\sqrt{m}-1\bigr)\\
&=2(h-1)-\frac{4\pi^2}{3}\frac{1}{\sqrt{m}}+O(m^{-1}).\nonumber
\end{align}
 This shows that $\m(S^{(l_0+2h)})\ge |\m_1(S^{(l_0+2h)})|>\RB(S^{(l_0+2h)})$ for $m\gg 0$.
 In fact, one can check that the right hand side of \eqref{for:estimate} is positive whenever $m\ge 39$.
\end{proof}

\begin{proof}
[Proof of Theorem~\ref{thm:oddmaxbound}]
 From Lemma~\ref{prop:nonRamanujanex}, we know that $X(S^{(l_0+4)})$ is not Ramanujan for $m\ge 39$.
 Moreover, one can see that the situations for $15\le m\le 37$ are the same as above
 by checking $|\m_1(S^{(l_0+4)})|>\RB(S^{(l_0+4)})$ individually.
\end{proof}

%\begin{remark}
% We can actually show that $\m(S^{(l)})=|\m^{1}(S^{(l)})|$ for all odd $l$.
% Namely, the inequality
% $\bigl|\frac{\sin{\frac{\pi l}{m}}}{\sin{\frac{\pi }{m}}}\bigr|\ge \bigl|\frac{\sin{\frac{\pi k l}{m}}}{\sin{\frac{\pi k}{m}}}\bigr|$
% hold for all $1\le k\le m-1$.
%\end{remark}

 We remark that the above discussion does not work for the case $h=1$, that is, $l=l_0+2$.
% in the determination of whether $X(S^{(l_0+2)})$ is Ramanujan or not.

\subsection{A criterion for ordinary $m$}

 From Theorem~\ref{thm:oddmaxbound},
 our task is to determine the number $\varepsilon\in\{0,2\}$ satisfying $\hat{l}=l_0+\varepsilon$ for a given $m$.
 Let us call $m$ {\it ordinary} if $\varepsilon=0$ and {\it exceptional} otherwise.
 This is based on the numerical fact that
 there are much more $m$ of the former type rather than the latter.
 The aim of this subsection is to give a criterion for ordinary $m$.

 Let $k\in\mathbb{Z}_{>0}$ and put $I_k=\{x\in\bR\,|\,\Gauss{\sqrt{x}-\frac{3}{2}}=k\}=[k^2+3k+\frac{9}{4},k^2+5k+\frac{25}{4})$.
 We now study an interpolation function $d(x)$ for the difference between $|\mu_{1}(S^{(l_0+2)})|$ and
 $\RB(S^{(l_0+2)})$ on $m\in I_k\cap (2\mathbb{Z}+1)$,
 that is, 
\[
 d(x)
=\frac{\sin{\frac{\pi (2k+3)}{x}}}{\sin{\frac{\pi}{x}}}-2\sqrt{x-2k-4}, \qquad x\in I_k.
\]
 Notice that $d(m)>0$ for $m\in I_k\cap (2\mathbb{Z}+1)$ implies that $X(S^{(l_0+2)})$ is not Ramanujan and hence $m$ is ordinary.
 Therefore, we are interested in the sign of the values of $d(x)$ on $I_{k}\cap (2\mathbb{Z}+1)$.
 The following lemma is crucial in our study.

\begin{lem}
\label{for:keylemma}
 Let $m\in I_k\cap (2\mathbb{Z}+1)$.
\begin{itemize}
 \item[$\mathrm{(1)}$] $d(m)<0$ for all $m\in I_k\cap (2\mathbb{Z}+1)$ when $k=1,2,3$.  
 \item[$\mathrm{(2)}$] $d(m)<0$ if and only if $m\in [k^2+5k-c,k^2+5k+5]$ with
\[
 c=
\begin{cases}
 3 & 4\le k\le 18, \\
% 4 &  9\le k\le 18,\\
 5 & k\ge 19.
% 6 & 49 \le k. 
\end{cases}
\]   
%\begin{itemize}
% \item[$\mathrm{(i)}$\ ] $k^2+5k-3\le m\le k^2+5k+6$ when $4\le k\le 8$,
% \item[$\mathrm{(ii)}$\,] $k^2+5k-4\le m\le k^2+5k+6$ when $9\le k\le 18$,
% \item[$\mathrm{(iii)}$] $k^2+5k-5\le m \le k^2+5k+6$ when $19\le k\le 48$,
% \item[$\mathrm{(iv)}$] $k^2+5k-6\le m \le k^2+5k+6$ when $k\ge 49$.
%\end{itemize} 
\end{itemize}
\end{lem} 
\begin{proof}
 The assertions for $k\le 8$ are direct.
 Let $k\ge 9$.
 We first claim that $d(x)$ is monotone decreasing on $I_k$.
 Actually,
 using the inequalities $x-\frac{x^3}{6}<\sin{x}<x$ and $1-\frac{x^2}{2}<\cos{x}<1-\frac{x^2}{2}+\frac{x^4}{24}$,
 we have
\begin{align*}
 d'(x)
&=-\frac{1}{\sqrt{x-2k-4}}
+\frac{\pi}{x^2(\sin{\frac{\pi}{x}})^2}\Biggl(-(2k+3)\cos{\frac{\pi (2k+3)}{x}}\sin{\frac{\pi}{x}}+\sin{\frac{\pi (2k+3)}{x}}\cos{\frac{\pi}{x}}\Biggr)\\
%&<-\frac{1}{\sqrt{x}}
%+\frac{\pi}{x^2}\frac{1}{\bigl(\frac{\pi}{x}-\frac{1}{6}\bigl(\frac{\pi}{x}\bigr)^3\bigr)^2}\Biggl(-C\Bigl(1-\frac{1}{2}\Bigl(\frac{\pi C}{x}\Bigr)^2\Bigr)
%\Bigl(\frac{\pi}{x}-\frac{1}{6}\Bigl(\frac{\pi}{x}\Bigr)^3\Bigr)
%+\frac{\pi C}{x}\Bigl(1-\frac{1}{2}\Bigl(\frac{\pi}{x}\Bigr)^2+\frac{1}{24}\Bigl(\frac{\pi}{x}\Bigr)^4\Bigr)\Biggr)\\
&<-\frac{1}{\sqrt{x}}+
\frac{\pi^2 (2k+3)(3(2k+3)^2-2)}{6x^3\bigl(1-\frac{1}{6}\bigl(\frac{\pi}{x}\bigr)^2\bigr)^2}\Biggl(1-\frac{\pi^2(2(2k+3)^2-1)}{4x^2(3(2k+3)^2-2)}\Biggr)\\
&<-\frac{1}{\sqrt{x}}+
\frac{\pi^2(2k+3)^3}{2x^3\bigl(1-\frac{1}{6}\bigl(\frac{\pi}{x}\bigr)^2\bigr)^2}.
\end{align*}
 Here we have clearly $2k+3<2(k+3)$, 
% since $x\in I_k=[k^2+3k+\frac{9}{4},k^2+5k+\frac{25}{4})$,
 $x>\pi$ and $k(k+3)<x<(k+3)^2$ for $x\in I_k$.
 Therefore,  
\begin{align*}
 d'(x)
&<-\frac{1}{\sqrt{x}}+\frac{144\pi^2(k+3)^3}{25x^3}
<-\frac{1}{k+3}+\frac{144\pi^2}{25}\frac{1}{k^3}<0
\end{align*}
% Notice that the last inequality follows because $k^3-\frac{144\pi^2}{25}(k+3)>0$ if
 for $k\ge 9$.
 This shows the assertion. 
 Next we investigate the value $D(k)=D(k,c)=d(k^2+5k+c)$
%, that is, 
%\begin{equation*}
% F_0(k)
%=\frac{\sin{\frac{\pi (2k+3)}{k^2+5k+c}}}{\sin{\frac{\pi}{k^2+5k+c}}}-2\sqrt{k^2+3k+c-4},
%\end{equation*}
 where $c\le 6$ is an integer not depending on $k$.
 It is easy to see that   
\[
 D(k)
=\frac{3c'-16\pi^2}{12}k^{-1}+O(k^{-2}),
\]
 where $c'=25-4c$.
 This shows that $D(k)<0$ for $k\gg 0$ if the leading coefficient is negative,
 that is, $c>\frac{75-16\pi^2}{12}=-6.90\ldots$.
 Actually, for $k\ge 49$, one can see that $D(k,-7)>0$ and $D(k,-6)<0$,
 whence, together with the monotoneness of $d(x)$, we obtain the desired claim for $k\ge 49$. 
 The rest of assertions, that is, for $9\le k\le 48$, are also checked individually.
 This completes the proof because $k^2+5k+c$ is odd if and only if $c$ is.
\end{proof}

 From this lemma, one can obtain a criterion for ordinary $m$.

\begin{thm}
\label{thm:criterionSl0+2}
 Let $m\ge 15$ be an odd integer.
 Put 
\[
 J=\bigl\{2n+1\,\bigr|\,7\le n\le 14\bigr\}\sqcup \bigsqcup_{c\in\{\pm 1,\pm 3,\pm 5\}} J_c,
\]
 where 
\[
 J_c
=
\begin{cases}
 \{k^2+5k+c\,|\,k\ge 4\} & c\in\{\pm 1,\pm 3,5\},\\
 \{k^2+5k-5\,|\,k\ge 19\} & c=-5.
\end{cases}
\]
 Then, $m$ is ordinary if $m\notin J$.
\end{thm}
\begin{proof}
% We first remark that $k^2+5k+c$ is odd if and only if $c$ is.
 Suppose that $m$ is not in $J$.
 Then, from Lemma~\ref{for:keylemma}, one sees that $d(m)>0$, in other words,
 $|\mu_{1}(S^{(l_0+2)})|>\RB(S^{(l_0+2)})$.
 This shows that $m$ is ordinary.
\end{proof}

 Theorem~\ref{thm:criterionSl0+2} implies that from now on we may concentrate only on $m$ with $m\in J$
 and leads us to imagine that the quadratic polynomials
\[
 f_c(k)=k^2+5k+c, \qquad c\in\{\pm 1,\pm 3,\pm 5\},
\] 
 play important roles in our study.
 We remark that the constant 
\[
 c'=25-4c>0,
\]
 which was in the proof of Lemma~\ref{for:keylemma},
 is nothing but the discriminant of $f_c(k)$ and will often appear in several arguments.

\section{Spectral consideration}

 In the subsequent discussion, 
 we only consider the case where $m\in J$, that is, 
 $m$ can be written as $m=f_c(k)=k^2+5k+c$
 for some $k\in\bZ_{>0}$ % ($k=\Gauss{\sqrt{m}-\frac{3}{2}}$)
 and $c\in \{\pm 1,\pm 3,\pm 5\}$.
% For such $m$, our problem is to decide whether $\hat{l}=l_0$ or $\hat{l}=l_0+2$ holds.
 For such $m$, we clarify when exceptionals occur.
 Hence, from now on, we concentrate on the circulant graphs $X(S)$ with $S\in\cS_{l_0+2}$.
 In this section, we use the notations $\RB=2\sqrt{m-(l_0+2)-1}$ and 
% For $l\in\mathcal{L}$, let  
\[
 \hat{\mu}=\max_{S\in\cS_{l_0+2}}\mu(S),
\]
 for simplicity.
% Namely, $\hat{\mu}_{l}$ is the largest non-trivial largest eigenvalues of all circulant graphs
% with fixed covalency $l$.
 From the definition, % it is immediate that
 $m$ is exceptional if and only if $\hat{\mu}\le \RB$. 
 Therefore, we have to decide $\hat{\mu}$ for a given $m$.

\subsection{A necessary condition for exceptionals}

 The aim of this subsection is to obtain the following necessary condition for exceptionals,
 which we can relatively easily reach the conclusion.

\begin{prop} 
\label{thm:primeinJ}
 Let $m\ge 15$ be an odd integer. 
 If $m$ is exceptional,
 then $m\in J$ which is in either of the following three types;
\begin{enumerate}
 \item[$(\I)$] $m=p$ is an odd prime. 
 \item[$(\II)$] $m=pq$ is a product of two odd primes $p$ and $q$ satisfying $p<q<4p$.
 \item[$(\III)$] $m=25,49$.
\end{enumerate}
% Moreover, every odd prime $m=p\in J$ and $m=25,49$ are indeed exceptional.
\end{prop}
\begin{proof}
 From Theorem~\ref{thm:criterionSl0+2},
 it is enough to consider only the case where $m\in J$.

 Assume that $m$ is a composite. 
 One can easily see that
 there are finitely many $m\in J$ such that $m=p^2$, that is, $m=25,49$.
 It is directly checked that these are all exceptional.
 For the other cases, let $p$ be the minimum prime factor of $m$ and write $m=pt$ with $3\le p<t$.
 If one can take 
 $S\in\cS_{l_0+2}$ as
 $\mathbb{Z}_m\setminus S\subset \{0,\pm p,\pm 2p,\ldots,\pm\frac{t-1}{2}p\}$,
 then $m$ is ordinary because   
\[
 |\mu_{t}(S)|
=\left|\sum_{b\in\mathbb{Z}_m\setminus S}e^{\frac{4\pi ib t}{m}}\right|
=l_0+2\ge \RB
\]
 from the definition of $l_0$. % that $X(S)$ is not Ramanujan and hence $m$ is ordinary.
 Such $S$ can be in fact taken if and only if
 $l_0+2=\#(\mathbb{Z}_m\setminus S)\le \#\{0,\pm p,\pm 2p,\ldots,\pm\frac{t-1}{2}p\}=t$, that is,
 $2\Gauss{\sqrt{pt}-\frac{3}{2}}+3\le t$, equivalently  $t\ge 4p-3$.
 Therefore, if $t$ is either composite or odd prime with $t\ge 4p-3$, then $m$ is ordinary.
% Here, one can respectively check that $m\in J$ of the form of both $m=p(4p-1)$ and $m=p(4p-3)$ where $p$ is odd prime, that is,
% $m=33$ and $m=27,85,451$, are all ordinary.
 This shows the claim.
\end{proof}

\begin{remark}
\label{rem:finitelyordinary}
 As we have stated in the above proof,
 the necessary condition $p<q<4p$ in ($\II$) can be actually reduced to $p<q\le 4p-5$.
 We also remark that 
 there are finitely many $m\in J$ of the form of both $m=p(4p-1)$ and $m=p(4p-3)$;
 $m=33$ and $m=27,85,451$, respectively.
 These are of course ordinary.
\end{remark}

\subsection{Exceptionals of type ($\boldsymbol{\I}$)}

 It is easy to see that $m\in J$ of type ($\I$) is actually exceptional.

\begin{thm} 
\label{thm:exceptionalI}
 Every odd prime $m=p\in J$ are exceptional.
\end{thm}
\begin{proof}
 Let $m=p\in J$ be a prime.
 Then, % for all $l\in\cL$,
 one can easily see that $\hat{\mu}=|\mu_{1}(S^{(l_0+2)})|$
 because the map $\bZ_m\to \bZ_m$ defined by $x\mapsto jx$ is bijective for all $1\le j\le m-1$.
% In particular, $\hat{\mu}_{l_0+2}=|\mu_{1}(S^{(l_0+2)})|$ and
 Hence, from Lemma~\ref{for:keylemma}, $m$ is exceptional if and only if $m\in J$.
% if and only if $f_0(m)\le 0$.
% Therefore, the claim follows from Lemma~\ref{for:keylemma}. % or Theorem~\ref{thm:criterionSl0+2}. 
\end{proof}

\subsection{Exceptionals of type ($\boldsymbol{\II}$)}

 In this subsection,
 we assume that $m=f_c(k)\in J$ is of type ($\II$).
 Namely, there exists odd distinct primes $p$ and $q$ with $p<q<4p$ such that $m=pq$.
 From Proposition~\ref{thm:primeinJ},
 our task is clear up whether or not such $m$ is in fact exceptional.
 We at first show that one can narrow down the candidates of $\hat{\mu}$ as follows.

\begin{lem}
\label{lem:candidatepq}
 We have $\hat{\mu}=\max\{\mu^{(0)},\mu^{(1)},\mu^{(2)}\}$ where
\begin{align}
\label{for:pq0}
 \mu^{(0)}
&=\frac{\sin{\frac{\pi(l_0+2)}{pq}}}{\sin{\frac{\pi}{pq}}},\\
\label{for:pq1}
 \mu^{(1)}
&=q+(l_0+2-q)\cos{\frac{2\pi}{p}},\\
\label{for:pq2}
  \mu^{(2)}
&=
 \begin{cases}
 p+(l_0+2-p)\cos{\frac{2\pi}{q}} & (p<q<\frac{(3p+2)^2}{4p}), \\
 p+2p\cos{\frac{2\pi}{q}}+(l_0+2-3p)\cos{\frac{4\pi}{q}} & (\frac{(3p+2)^2}{4p}\le q<4p).
\end{cases}
\end{align}
\end{lem}
\begin{proof}
 From the definition, we have
\begin{align*}
 \hat{\mu}
%=\max_{S\in\cS_{l_0+2}}\mu(S)
=\max_{S\in\cS_{l_0+2}}\Bigl\{\max_{1\le j\le pq-1}|\mu_j(S)|\Bigr\}
=\max\bigl\{\mu^{(0)},\mu^{(1)},\mu^{(2)}\bigr\},
\end{align*}
 where  
\begin{align*}
 \mu^{(0)}&=\max_{S\in\cS_{l_0+2}}\Bigl\{\max_{1\le j\le pq-1 \atop (j,pq)=1}|\mu_j(S)|\Bigr\},\\
 \mu^{(1)}&=\max_{S\in\cS_{l_0+2}}\Bigl\{\max_{1\le j\le pq-1 \atop (j,pq)=q}|\mu_j(S)|\Bigr\},\\
 \mu^{(2)}&=\max_{S\in\cS_{l_0+2}}\Bigl\{\max_{1\le j\le pq-1 \atop (j,pq)=p}|\mu_j(S)|\Bigr\}.
\end{align*}
 Hence, it is enough to show that $\mu^{(0)},\mu^{(1)},\mu^{(2)}$ are equal to 
 \eqref{for:pq0}, \eqref{for:pq1}, \eqref{for:pq2}, respectively.
 The expression \eqref{for:pq0}, that is, $\mu^{(0)}=|\mu_{1}(S^{(l_0+2)})|$,
 can be seen similarly as in the proof of Theorem~\ref{thm:primeinJ}.
 Therefore, it suffices to consider the other two cases. 
 
 For $\mu^{(1)}$, we have
\begin{align}
\label{for:mu1exp}
 \mu^{(1)}
&=\max_{S\in\cS_{l_0+2}}
\left\{\max_{1\le s\le p-1}\Biggl|\sum^{p-1}_{t=0}\#\bigl\{b\in\bZ_{pq}\,\bigl|\,b\notin S,\ b\equiv t \!\!\!\pmod{p}\bigr\}e^{\frac{2\pi i ts}{p}}\Biggr|\right\}.
\end{align}
 Now, we introduce the notation 
\[
 T_{h}(a,b)
=
\begin{cases}
 \bigl\{0,\pm a,\pm 2a,\ldots,\pm\frac{b-1}{2}a\bigr\} & (h=0),\\
 \bigl\{\pm h,\pm a\pm h,\pm 2a\pm h,\ldots,\pm \frac{b-1}{2}a\pm h\bigr\} & (h\ge 1)
\end{cases}
\]
 for odd $a,b\in\mathbb{Z}_{>0}$.
 Note that $\# T_{h}(a,b)=b$ if $h=0$ and $2b$ otherwise.
 We may assume that $p<q\le 4p-5$ (see Remark~\ref{rem:finitelyordinary}).
 This condition implies that we can not take $S\in\cS_{l_0+2}$ as $\bZ_{pq}\setminus S\subset T_{0}(p,q)$
 but can as $\bZ_{pq}\setminus S\subset T_{0}(p,q)\cup T_{1}(p,q)$.
 This together with \eqref{for:mu1exp} shows the expression \eqref{for:pq1}.
% (the maximum is attained for such $S$ with $s=1$).

% On the other hand, when $p\,|\,j$, one has similarly
 Similarly, we have 
\begin{align}
\label{for:mu2exp}
\mu^{(2)}
&=\max_{S\in\cS_{l_0+2}}
\left\{\max_{1\le s\le q-1}\Biggl|\sum^{q-1}_{t=0}\#\bigl\{b\in\bZ_{pq}\,\bigl|\,b\notin S,\ b\equiv t \!\!\!\pmod{q}\bigr\}e^{\frac{2\pi i ts}{q}}\Biggr|\right\}.
\end{align}
 The condition $p<q$ implies that we can not take $S\in\cS_{l_0+2}$ as $\bZ_{pq}\setminus S\subset T_{0}(q,p)$.
 However, if $l_0+2\le 3p$, that is, $2\Gauss{\sqrt{pq}-\frac{3}{2}}+3\le 3p$, equivalently
 $(p,q)=(3,5),(3,7)$ or $p<q<\frac{(3p+2)^2}{4p}$ if $p\ge 5$,
% (the equality holds if and only if
% $q=\frac{9p+7}{4},\frac{9p+11}{4}$ if $p\equiv 1$ $(\text{mod} \ 4)$
% or $q=\frac{9p+1}{4},\frac{9p+5}{4}$ if $p\equiv 3$ $(\text{mod} \ 4)$),
 then we can take $S\in\mathcal{S}_{l_0+2}$ as $\bZ_{pq}\setminus S\subset T_{0}(q,p)\cup T_{1}(q,p)$ and hence,
 together with \eqref{for:mu2exp}, $\mu^{(2)}=p+(l_0+2-p)\cos{\frac{2\pi}{q}}$.
 If $\frac{(3p+2)^2}{4p}\le q<4p$, then we can not take $S\in\cS_{l_0+2}$ as $\bZ_{pq}\setminus S\subset T_{0}(q,p)\cup T_{1}(q,p)$
 but can as $\bZ_{pq}\setminus S\subset T_{0}(q,p)\cup T_{1}(q,p)\cup T_{2}(q,p)$,
 whence $\mu^{(2)}=p+2p\cos{\frac{2\pi}{q}}+(l_0+2-3p)\cos{\frac{4\pi}{q}}$.
 These show the expression \eqref{for:pq2}.  
\end{proof}

 We next analytically evaluate the difference between $\mu^{(i)}$ and $\RB$ on $J\cap I_k$ for each $i\in\{0,1,2\}$.
 Before that, we notice that when $m=f_c(k)=pq\in J\cap I_k$ is of type ($\II$), we have $l_0+2=2k+3$.
 Moreover, if we put $x=\sqrt{\frac{q}{p}}$, then $1<x<2$ and $p=\frac{\sqrt{f_c(k)}}{x}$ and $q=\sqrt{f_c(k)}x$.
 Based on these facts, we study the following functions
\[
 D^{(i)}_c(k;x)=M^{(i)}_c(k;x)-A, \qquad k\in\mathbb{Z}_{>0}, \ 1<x<2,
\] 
 where $M^{(i)}_c(k;x)$ is defined by 
\begin{align*}
 M^{(0)}_c(k;x)
&=\frac{\sin{\frac{\pi C}{B^2}}}{\sin{\frac{\pi}{B^2}}},\\
 M^{(1)}_c(k;x)
&=Bx+\Bigl(C-Bx\Bigr)\cos{\frac{2\pi x}{B}},\\
 M^{(2)}_c(k;x)
&=
\begin{cases}
 \displaystyle{\frac{B}{x}+\Bigl(C-\frac{B}{x}\Bigr)\cos{\frac{2\pi}{Bx}}} & (1<x<\frac{3}{2}),\\[7pt]
 \displaystyle{\frac{B}{x}+\frac{2B}{x}\cos{\frac{2\pi}{Bx}}+\Bigl(C-\frac{3B}{x}\Bigr)\cos{\frac{4\pi}{Bx}}} & (\frac{3}{2}<x<2),
\end{cases}
\end{align*}
 with 
\begin{align*}
 A&=2\sqrt{f_c(k)-(l_0+2)-1}=2\sqrt{k^2+3k+c-4},\\
 B&=\sqrt{f_c(k)}=\sqrt{k^2+5k+c},\\
 C&=l_0+2=2k+3.
\end{align*}

\begin{lem}
\label{lem:Mi}
 For a fixed $x$, we have
\begin{align}
\label{for:M0}
 M^{(0)}_c(k;x)
&=2k+3-\frac{4\pi^2}{3}k^{-1}+O(k^{-2}),\\
\label{for:M1}
 M^{(1)}_c(k;x)
&=2k+3-2\pi^2x^2(2-x)k^{-1}+O(k^{-2}),\\
\label{for:M2}
 M^{(2)}_c(k;x)
&=
\begin{cases}
 \displaystyle{2k+3-\frac{2\pi^2x^2(2x-1)}{x^3}k^{-1}+O(k^{-2})} & (1<x<\frac{3}{2}),\\[7pt]
 \displaystyle{2k+3-\frac{4\pi^2x^2(4x-5)}{x^3}k^{-1}+O(k^{-2})} & (\frac{3}{2}<x<2),
\end{cases}
\end{align}
 and
\begin{equation}
\label{for:A}
A=2k+3-\frac{c'}{4}k^{-1}+O(k^{-2})
\end{equation}
 as $k\to\infty$.
\end{lem}
\begin{proof}
 These are direct.
\end{proof}

 We first show that 
 one does not have to take account of both $D^{(0)}_c$ and $D^{(2)}_c$
 in our discussion.

\begin{lem}
\label{lem:m02}
 We have $D^{(0)}_c(k;x)<0$ and $D^{(2)}_c(k;x)<0$ on $1<x<2$
 for any $c\in\{\pm 1,\pm 3,\pm 5\}$ and $k\ge k_0$ with a sufficiently large $k_0\in\mathbb{N}$.
\end{lem}
\begin{proof}
 Notice that $D^{(0)}_c(k;x)=D(k)$ where $D(k)=D(k,c)$ is defined in the proof of Lemma~\ref{for:keylemma}.
 Therefore, as we have seen in the lemma, 
 $D^{(0)}_c(k;x)=D(k)<0$ because $c\in\{\pm 1,\pm 3,\pm 5\}$ if $k$ is sufficiently large.

 Similarly, from \eqref{for:M2} and \eqref{for:A}, we have 
\[
 D^{(2)}_c(k;x)=
\begin{cases}
 \displaystyle{\frac{c'x^3-8\pi^2(2x-1)}{4x^3}}k^{-1}+O(k^{-2}) & (1<x<\frac{3}{2}),\\[7pt]
 \displaystyle{\frac{c'x^3-16\pi^2(4x-5)}{4x^3}}k^{-1}+O(k^{-2}) & (\frac{3}{2}<x<2).
\end{cases}
\]
 Hence the result follows from the fact that the coefficient of $k^{-1}$ is negative for any $1<x<2$.
\end{proof}

 For the function $D^{(1)}_c$, we have the asymptotic expansion
\begin{equation}
\label{for:asymptoticF1}
 D^{(1)}_c(k;x)=\frac{c'-8\pi^2x(2-x)}{4}k^{-1}+O(k^{-2})
\end{equation}
 from \eqref{for:M1} and \eqref{for:A}, and thus 
 $D^{(1)}_c$ becomes negative only if $x$ is close to $2$,
 in other words, $q$ is close to $4p$.
% (see figures in Remark~\ref{rem:spectralordering} again).
 More precisely, we obtain the following 

\begin{lem}
\label{lem:m1} 
 There exists constants $x_1$ and $x_2$ with $1<x_1<x_2<2$
% which are not depend on neither $k$ nor $c\in\{\pm 1,\pm 3,\pm 5\}$,
 such that % for any $k\ge k_0$ and $c\in\{\pm 1,\pm 3,\pm 5\}$
\begin{itemize}
 \item[$\mathrm{(1)}$] $D^{(1)}_c(k;x)<0$ on $1<x<x_1$, 
 \item[$\mathrm{(2)}$] $D^{(1)}_c(k;x)>0$ on $x_2<x<2$, 
\end{itemize}
 for any $c\in\{\pm 1,\pm 3,\pm 5\}$ and $k\ge k_0$ with a sufficiently large $k_0\in\mathbb{N}$.
\end{lem}
\begin{proof}
 We write $D^{(1)}_c(k;x)=B(1-\cos{\frac{2\pi x}{B}})(x-X(x;k,c))$ with 
\[
 X(x;k,c)=\frac{C}{B}-\frac{C-A}{B(1-\cos{\frac{2\pi x}{B}})}.
\]
%\begin{align*}
% F^{(1)}(x)
%&=(Bx+(C-Bx)\cos{\frac{2\pi x}{B}})-A\\
%&=C\cos{\frac{2\pi x}{B}}+Bx(1-\cos{\frac{2\pi x}{B}})-A\\
%&=-C(1-\cos{\frac{2\pi x}{B}})+Bx(1-\cos{\frac{2\pi x}{B}})+C-A\\
%&=B\Bigl(1-\cos{\frac{2\pi x}{B}}\Bigr)\Bigl(x-\frac{C}{B}+\frac{C-A}{B(1-\cos{\frac{2\pi x}{B}})}\Bigr)
%\end{align*}
 Since $B(1-\cos{\frac{2\pi x}{B}})>0$, $D^{(1)}_c(k;x)<0$ if and only if $x<X(x;k,c)$.
%\begin{equation}
%\label{def:X}
% x<\frac{C}{B}-\frac{C-A}{B(1-\cos{\frac{2\pi x}{B}})}=X(x;k,c).
%\end{equation}
 Hence, noticing that $C>A$ and $\cos\frac{2\pi x}{B}$ is monotone decreasing for $1<x<2$ when $k\ge 3$, % because $C^2-A^2=c'>0$,
 one finds that 
 if $1<x<X(1;k,c)$ (resp. $X(2;k,c)<x<2$), then $D^{(1)}_c(k;x)<0$ (resp. $D^{(1)}_c(k;x)>0$). 
 Here, from the expansion
\[
 X(x;k,c)=2-\frac{c'}{8\pi^2 x^2}+O(k^{-1}),
\]
 for a given $\varepsilon>0$, 
 there exists $k(\varepsilon;x,c)\in\mathbb{N}$ such that for any $k\ge k(\varepsilon;x,c)$
 we have $2-\frac{c'}{8\pi^2 x^2}-\varepsilon<X(x;k,c)<2-\frac{c'}{8\pi^2 x^2}+\varepsilon$.
 This implies that for any $k\ge \max\{k(\varepsilon;1,c),k(\varepsilon;2,c)\}$
 we have $2-\frac{c'}{8\pi^2}-\varepsilon< X(1;k,c)$ and $X(2;k,c)<2-\frac{c'}{32\pi^2}+\varepsilon$.
 Therefore, we can take $x_1$ and $x_2$ in the assertion as
\begin{align*}
 x_1=&\min_{c\in\{\pm 1,\pm 3,\pm 5\}}\overline{x}_1(c)=\overline{x}_1(-5)=1.4300\cdots,\\
 x_2=&\max_{c\in\{\pm 1,\pm 3,\pm 5\}}\underline{x}_2(c)=\underline{x}_2(5)=1.9841\cdots,
\end{align*}
 where $\overline{x}_1(c)=2-\frac{c'}{8\pi^2}$ and $\underline{x}_2(c)=2-\frac{c'}{32\pi^2}$.
% then they possess the desired properties.
\end{proof}

 Now we state the main result in this subsection,
 which follows immediately from Lemma~\ref{lem:m1}.  

\begin{thm}
\label{thm:pqinJallc}
 There exists constants $\xi_1$ and $\xi_2$ with $1<\xi_1<\xi_2<4$
% which are not depend on neither $k$ nor $c\in\{\pm 1,\pm 3,\pm 5\}$,
 such that, for sufficiently large $m=pq$ of type $(\II)$,
\begin{itemize}
 \item[$\mathrm{(1)}$] $m$ is exceptional if $1<\frac{q}{p}<\xi_1$.
 \item[$\mathrm{(2)}$] $m$ is ordinary if $\xi_2<\frac{q}{p}<4$.
\end{itemize}
\qed
\end{thm}
%\begin{proof}
% Take $\xi_i=x_i^2$.
% Take $\xi_1=\overline{x}_1(-5)^2=2.0451\ldots$ and $\xi_2=\underline{x}_2(5)^2=3.9369\ldots$, respectively. 
% If $1<\frac{q}{p}<\xi_1$, then, from Lemma~\ref{lem:m1}, we have $\mu^{(1)}<\RB_{l_0+2}$ and hence,
% together with Lemma~\ref{lem:m02}, 
% $\hat{\mu}_{l_0+2}=\max\{\mu^{(0)},\mu^{(1)},\mu^{(2)}\}<\RB_{l_0+2}$.
% This shows that $m$ is exceptional.
% On the other hand, if $\xi_2<\frac{q}{p}<4$, then, from Lemma~\ref{lem:m1} again,
% we have $\mu^{(1)}>\RB_{l_0+2}$ and hence $m$ is ordinary.
%\end{proof}

\begin{remark}
\label{rem:spectralordering}
 From Lemma~\ref{lem:Mi}, 
 one can find the asymptotic order of $\mu^{(0)}$, $\mu^{(1)}$, $\mu^{(2)}$ and $\RB$.
 Actually, the expansions \eqref{for:M0}, \eqref{for:M1}, \eqref{for:M2} and \eqref{for:A} assert
\begin{align*}
%\label{for:specorder}
\begin{cases}
 \mu^{(1)}<\mu^{(2)}<\mu^{(0)}<\RB_{l_0+2} & (1<x<\gamma_1),\\
 \mu^{(1)}<\mu^{(0)}<\mu^{(2)}<\RB_{l_0+2} & (\gamma_1<x<\gamma_2),\\
 \mu^{(1)}<\mu^{(2)}<\mu^{(0)}<\RB_{l_0+2} & (\gamma_2<x<\gamma_3),\\
 \mu^{(2)}<\mu^{(1)}<\mu^{(0)}<\RB_{l_0+2} & (\gamma_3<x<\gamma_4),\\
 \mu^{(2)}<\mu^{(0)}<\mu^{(1)}<\RB_{l_0+2} & (\gamma_4<x<\gamma_5(c)),\\
 \mu^{(2)}<\mu^{(0)}<\RB_{l_0+2}<\mu^{(1)} & (\gamma_5(c)<x<2),
\end{cases}
\end{align*}
 for sufficiently large $k>0$.
 Here $\gamma_1$, $\gamma_2$, $\gamma_3$, $\gamma_4$ and $\gamma_5(c)$ 
 are the real roots in the interval $(1,2)$ of the equations
 $2x^3-6x+3=0$,
 $x^3-12x+15=0$,
 $x^6-2x^5+8x-10=0$,
 $3x^3-6x^2+2=0$ and 
 $8\pi^2x^3-16\pi^2x^2+c'=0$, respectively.
 Remark that % let $\overline{x}_1(c):=2-\frac{c'}{8\pi^2}$ and $\underline{x}_2(c):=2-\frac{c'}{32\pi^2}$.
 one can numerically check the inequality $\overline{x}_1(c)<\gamma_5(c)<\underline{x}_2(c)$ as the table below,
 where $\overline{x}_1(c)=2-\frac{c'}{8\pi^2}$ and $\underline{x}_2(c)=2-\frac{c'}{32\pi^2}$ are defined in the proof of
 Lemma~\ref{lem:m1}.

\begin{table}[htbp]
\begin{center} 
{\renewcommand\arraystretch{1.2}
\begin{tabular}{c||c|c|c|c|c|c|c}
 $c$ & $\gamma_1$ & $\gamma_2$ & $\gamma_3$ & $\gamma_4$ & $\overline{x}_1(c)$ & $\gamma_5(c)$ &
 $\underline{x}_2(c)$  \\
\hline
\hline
 $-5$ &
 \multirow{6}{43pt}{1.3843\ldots} &
 \multirow{6}{43pt}{1.5765\ldots} &
 \multirow{6}{43pt}{1.7579\ldots} &
 \multirow{6}{43pt}{1.7925\ldots} &
 $1.4300\ldots$ & $1.8297\ldots$ & $1.8575\ldots$ \\
\cline{1-1}\cline{6-8}
 $-3$ & & & & & $1.5313\ldots$ & $1.8653\ldots$ & $1.8828\ldots$ \\
\cline{1-1}\cline{6-8}
 $-1$ & & & & & $1.6327\ldots$ & $1.8980\ldots$ & $1.9081\ldots$ \\
\cline{1-1}\cline{6-8}
 $1$ & & & & & $1.7340\ldots$ & $1.9284\ldots$ & $1.9335\ldots$ \\
\cline{1-1}\cline{6-8}
 $3$ & & & & & $1.8353\ldots$ & $1.9570\ldots$ & $1.9588\ldots$ \\
\cline{1-1}\cline{6-8}
 $5$ & & & & & $1.9366\ldots$ & $1.9839\ldots$ & $1.9841\ldots$ 
\end{tabular} 
}
\end{center} 
 \caption{The explicit values of $\gamma_1$, $\gamma_2$, $\gamma_3$, $\gamma_4$, $\overline{x}_1(c)$, $\gamma_5(c)$ and $\underline{x}_2(c)$.}
\end{table}
 See Figure~1-6 which show actual values of $\mu^{(0)}$, $\mu^{(1)}$, $\mu^{(2)}$ and $\RB$
 for $m=f_c(k)=pq\in J$ with $k=10^4$ for each $c\in\{\pm 1,\pm 3,\pm 5\}$,
 where the horizontal axis shows $x=\sqrt{\frac{q}{p}}$ and 
 the left and right vertical dashed lines describe $\overline{x}_1(c)$ and $\underline{x}_2(c)$, respectively.
 As we have seen in \eqref{for:asymptoticF1},
 the inequality $\mu^{(1)}>\RB$ holds when $x$ is very close to $2$.

\begin{figure}[htbp]
\begin{center}
\begin{tabular}{ccc}
  \begin{minipage}{0.33\textwidth}
    \begin{center}
     \includegraphics[clip,width=50mm]{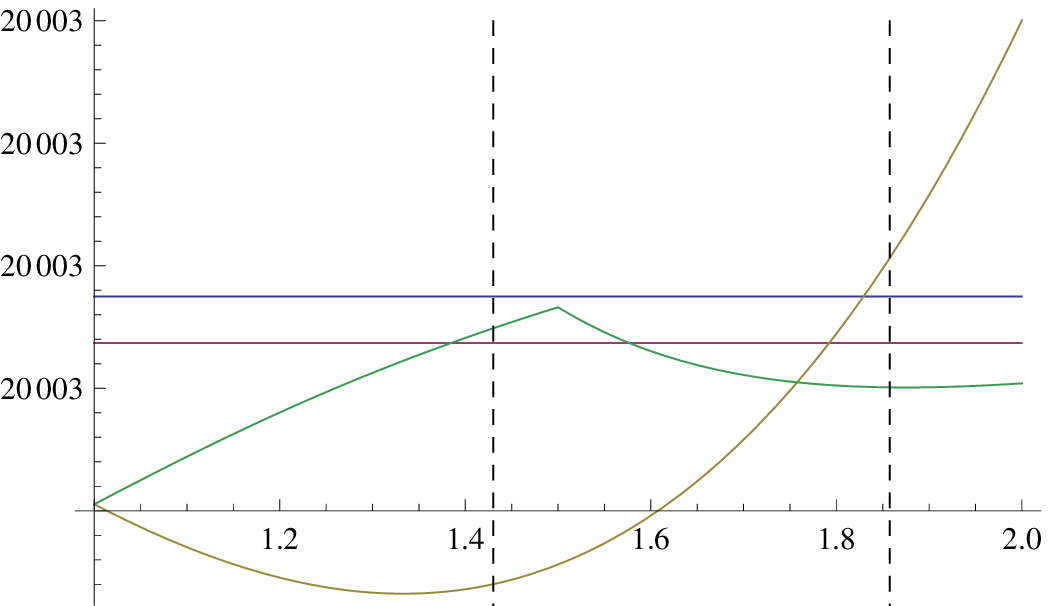}
     \caption{$c=-5$.}
    \end{center}
   \end{minipage}
   \begin{minipage}{0.33\textwidth}
    \begin{center}
     \includegraphics[clip,width=50mm]{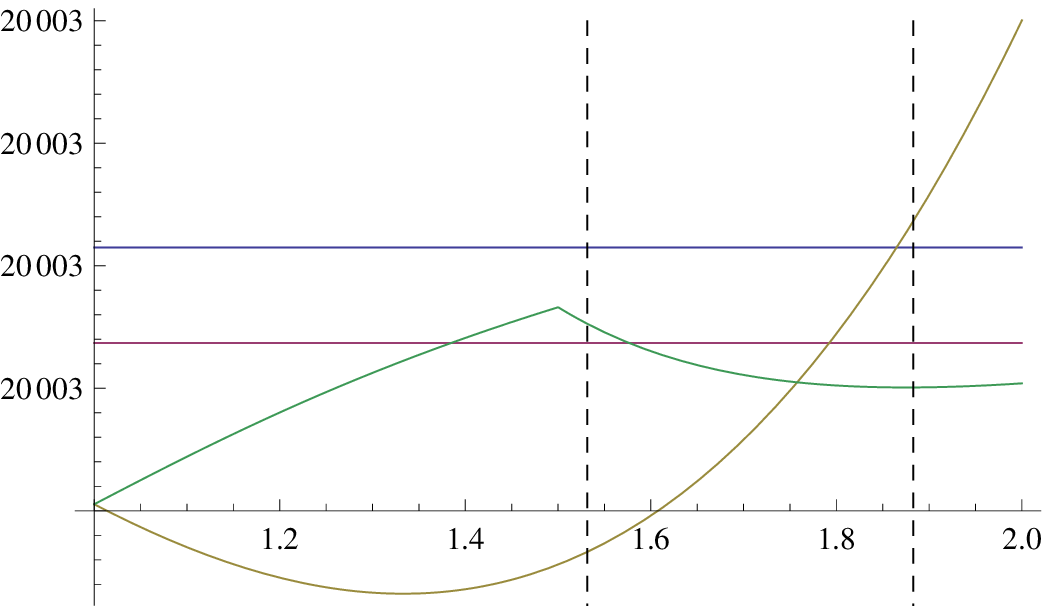}
     \caption{$c=-3$.}
    \end{center}
   \end{minipage}
      \begin{minipage}{0.33\textwidth}
    \begin{center}
     \includegraphics[clip,width=50mm]{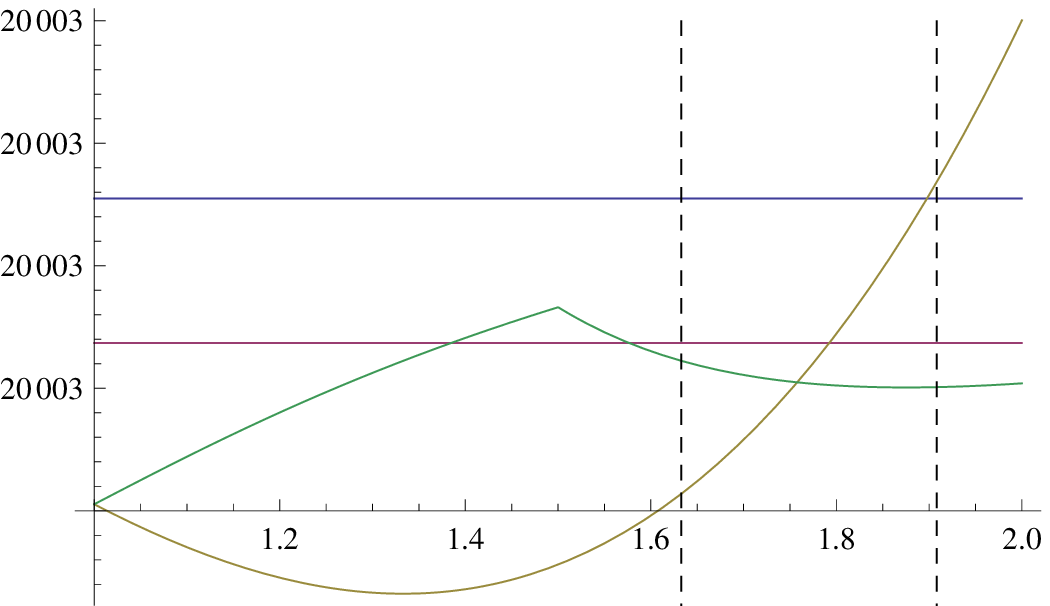}
     \caption{$c=-1$.}
    \end{center}
   \end{minipage}
  \end{tabular}
\end{center}

\begin{center}
\begin{tabular}{ccc}
  \begin{minipage}{0.33\textwidth}
    \begin{center}
     \includegraphics[clip,width=50mm]{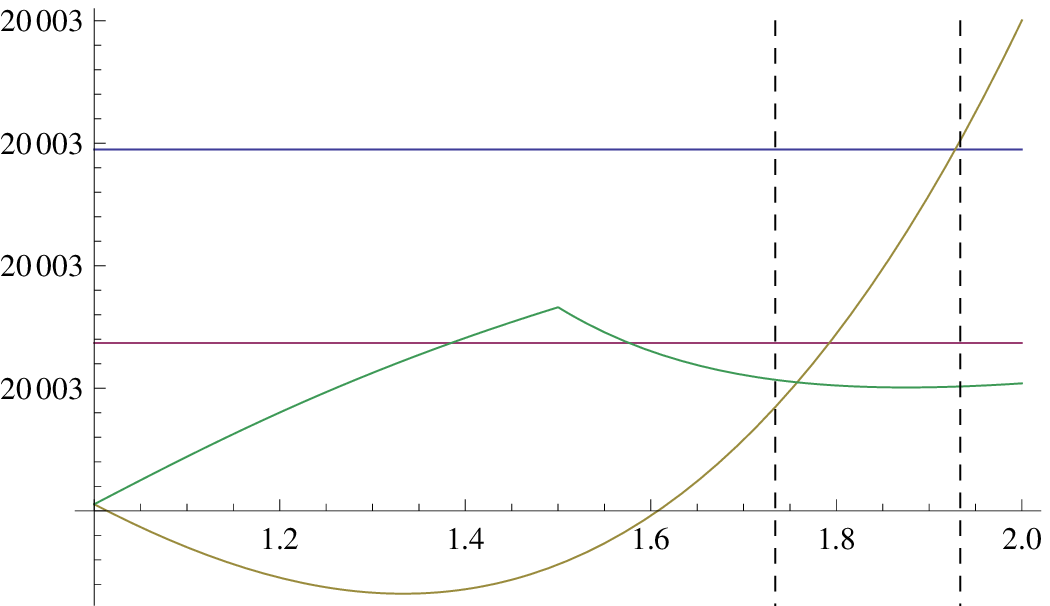}
     \caption{$c=1$.}
    \end{center}
   \end{minipage}
   \begin{minipage}{0.33\textwidth}
    \begin{center}
     \includegraphics[clip,width=50mm]{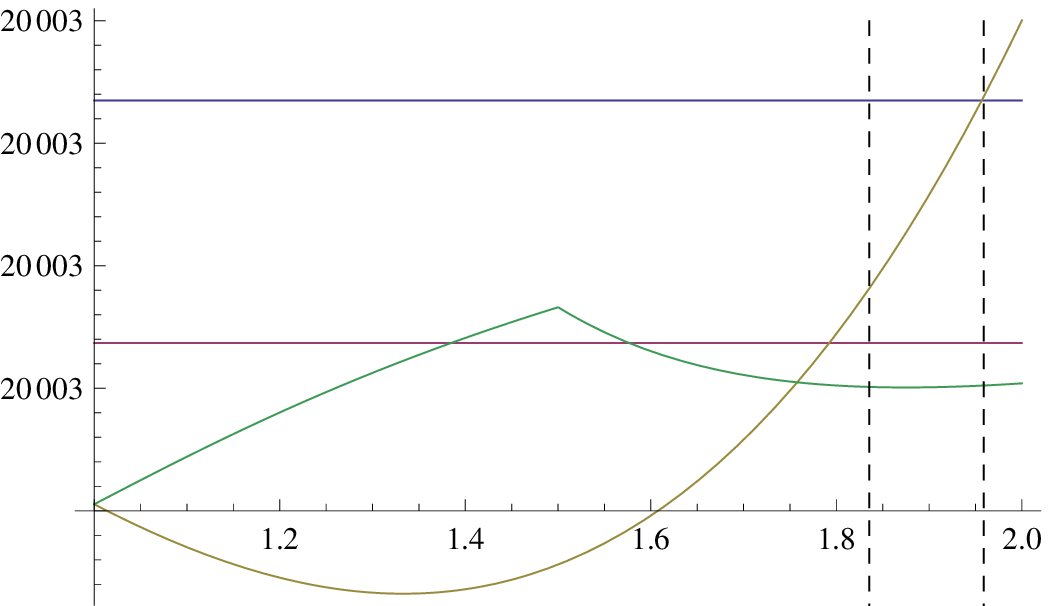}
     \caption{$c=3$.}
    \end{center}
   \end{minipage}
      \begin{minipage}{0.33\textwidth}
    \begin{center}
     \includegraphics[clip,width=50mm]{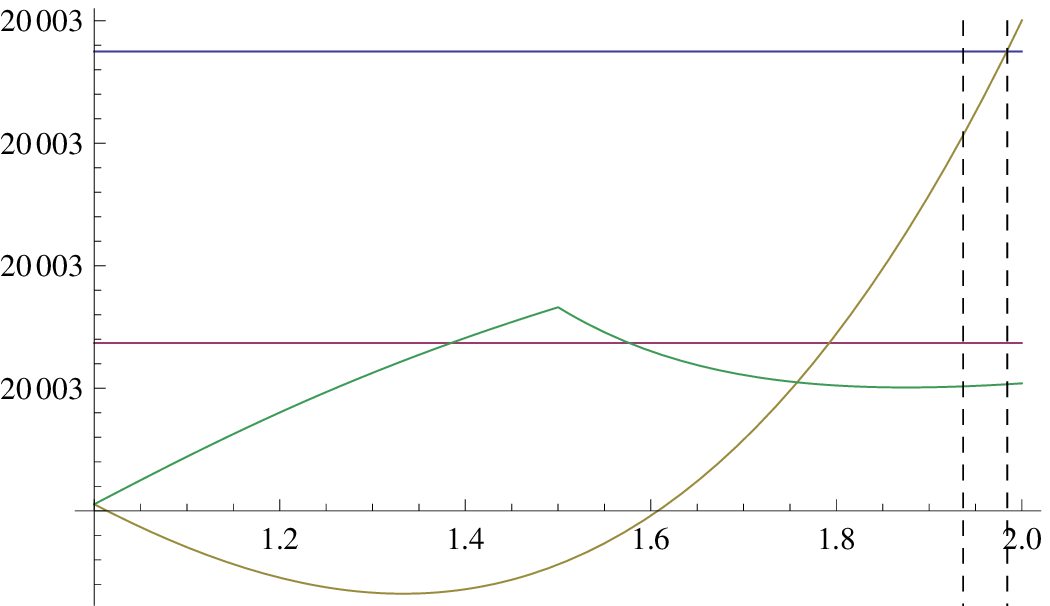}
     \caption{$c=5$.}
    \end{center}
   \end{minipage}
  \end{tabular}
\end{center}
\end{figure}

\end{remark}

\newpage 

\begin{table}[htbp]
\begin{center}
{\small 
{\renewcommand{\arraystretch}{1.0}
\begin{tabular}{c||c|c|c|c|c|c}
 $k$ & $c=-5$ & $c=-3$ & $c=-1$ & $c=1$ & $c=3$ & $c=5$ \\
\hline
\hline
%0&	-&	-&	-&	-&	\blue{\bf 3}&	\blue{\bf 5}\\ \hline
%1&	-&	-&	-&	\blue{\bf 7}&	{\bf 9}&	\blue{\bf 11}\\ \hline
%2&	-&	-&	\blue{\bf 13}&	\red{\bf 15}&	\blue{\bf 17}&	\blue{\bf 19}\\ \hline
%3&	-&	21&	\blue{\bf 23}&	{\bf 25}&	27&	\blue{\bf 29}\\ \hline
4&	-&	33&	\red{\bf 35}&	\blue{\bf 37}&	39&	\blue{\bf 41}\\ \hline
5&	-&	\blue{\bf 47}&	{\bf 49}&	51&	\blue{\bf 53}&	\red{\bf 55}\\ \hline
6&	-&	63&	\red{\bf 65}&	\blue{\bf 67}&	69&	\blue{\bf 71}\\ \hline
7&	-&	81&	\blue{\bf 83}&	85&	87&	\blue{\bf 89}\\ \hline
8&	-&	\blue{\bf 101}&	\blue{\bf 103}&	105&	\blue{\bf 107}&	\blue{\bf 109}\\ \hline
9&	-&	123&	125&	\blue{\bf 127}&	129&	\blue{\bf 131}\\ \hline
10&	-&	147&	\blue{\bf 149}&	\blue{\bf 151}&	153&	155\\ \hline
11&	-&	\blue{\bf 173}&	175&	177&	\blue{\bf 179}&	\blue{\bf 181}\\ \hline
12&	-&	201&	203&	205&	207&	\red{\bf 209}\\ \hline
13&	-&	231&	\blue{\bf 233}&	235&	237&	\blue{\bf 239}\\ \hline
14&	-&	\blue{\bf 263}&	265&	267&	\blue{\bf 269}&	\blue{\bf 271}\\ \hline
15&	-&	297&	\red{\bf 299}&	301&	303&	305\\ \hline
16&	-&	333&	335&	\blue{\bf 337}&	339&	\red{\bf 341}\\ \hline
17&	-&	371&	\blue{\bf 373}&	375&	\red{\bf 377}&	\blue{\bf 379}\\ \hline
18&	-&	411&	413&	415&	417&	\blue{\bf 419}\\ \hline
19&	451&	453&	455&	\blue{\bf 457}&	459&	\blue{\bf 461}\\ \hline
20&	495&	497&	\blue{\bf 499}&	501&	\blue{\bf 503}&	505\\ \hline
21&	\blue{\bf 541}&	543&	545&	\blue{\bf 547}&	549&	\red{\bf 551}\\ \hline
22&	\red{\bf 589}&	591&	\blue{\bf 593}&	595&	597&	\blue{\bf 599}\\ \hline
23&	639&	\blue{\bf 641}&	\blue{\bf 643}&	645&	\blue{\bf 647}&	649\\ \hline
24&	\blue{\bf 691}&	693&	695&	\red{\bf 697}&	699&	\blue{\bf 701}\\ \hline
25&	745&	747&	749&	\blue{\bf 751}&	753&	755\\ \hline
26&	801&	803&	805&	807&	\blue{\bf 809}&	\blue{\bf 811}\\ \hline
27&	\blue{\bf 859}&	861&	\blue{\bf 863}&	865&	867&	869\\ \hline
28&	\blue{\bf 919}&	921&	923&	925&	927&	\blue{\bf 929}\\ \hline
29&	981&	\blue{\bf 983}&	985&	987&	\red{\bf 989}&	\blue{\bf 991}\\ \hline
30&	1045&	1047&	\blue{\bf 1049}&	\blue{\bf 1051}&	1053&	1055\\ \hline
31&	1111&	1113&	1115&	\blue{\bf 1117}&	1119&	\red{\bf 1121}\\ \hline
32&	1179&	\blue{\bf 1181}&	1183&	1185&	\blue{\bf 1187}&	\red{\bf 1189}\\ \hline
33&	\blue{\bf 1249}&	1251&	1253&	1255&	1257&	\blue{\bf 1259}\\ \hline
34&	\blue{\bf 1321}&	1323&	1325&	\blue{\bf 1327}&	1329&	1331\\ \hline
35&	1395&	1397&	\blue{\bf 1399}&	1401&	\red{\bf 1403}&	1405\\ \hline
36&	\blue{\bf 1471}&	1473&	1475&	1477&	1479&	\blue{\bf 1481}\\ \hline
37&	\blue{\bf 1549}&	1551&	\blue{\bf 1553}&	1555&	1557&	\blue{\bf 1559}\\ \hline
38&	1629&	1631&	\red{\bf 1633}&	1635&	\blue{\bf 1637}&	1639\\ \hline
39&	\red{\bf 1711}&	1713&	1715&	1717&	1719&	\blue{\bf 1721}\\ \hline
40&	1795&	1797&	1799&	\blue{\bf 1801}&	1803&	1805\\ \hline
41&	1881&	1883&	1885&	1887&	\blue{\bf 1889}&	\red{\bf 1891}\\ \hline
42&	1969&	1971&	\blue{\bf 1973}&	1975&	1977&	\blue{\bf 1979}\\ \hline
43&	\red{\bf 2059}&	2061&	\blue{\bf 2063}&	2065&	2067&	\blue{\bf 2069}\\ \hline
44&	2151&	\blue{\bf 2153}&	2155&	2157&	2159&	\blue{\bf 2161}\\ \hline
45&	2245&	2247&	2249&	\blue{\bf 2251}&	2253&	2255\\ \hline
46&	\blue{\bf 2341}&	2343&	2345&	\blue{\bf 2347}&	2349&	\blue{\bf 2351}\\ \hline
47&	2439&	\blue{\bf 2441}&	2443&	2445&	\blue{\bf 2447}&	\red{\bf 2449}\\ \hline
48&	\blue{\bf 2539}&	2541&	\blue{\bf 2543}&	2545&	2547&	\blue{\bf 2549}\\ \hline
49&	2641&	2643&	2645&	\blue{\bf 2647}&	2649&	2651\\ \hline
50&	2745&	\red{\bf 2747}&	\blue{\bf 2749}&	2751&	\blue{\bf 2753}&	2755 %\\ \hline
%51&	\blue{\bf 2851}&	2853&	2855&	\blue{\bf 2857}&	2859&	\blue{\bf 2861}\\ \hline
%52&	2959&	2961&	\blue{\bf 2963}&	2965&	2967&	\blue{\bf 2969}\\ \hline
%53&	3069&	\red{\bf 3071}&	3073&	3075&	3077&	\blue{\bf 3079}\\ \hline
%54&	3181&	\blue{\bf 3183}&	3185&	\blue{\bf 3187}&	3189&	\blue{\bf 3191}\\ \hline
%55&	3295&	3297&	\blue{\bf 3299}&	\blue{\bf 3301}&	3303&	3305
\end{tabular}
}
}
 \caption{List of small exceptionals $m=p\in J$ of type ($\I$) (blue bold numbers)
 and $m=pq\in J$ with $p<q<4p$ of type $(\II)$ (red bold numbers).}
\end{center}
\end{table}

\section{Arithmetic consideration}

 Let $m\ge 15$.
 Then, $m$ is one of the followings; type ($\I$), ($\II$) and the others. 
 Remark that, from Theorem~\ref{thm:criterionSl0+2} and Proposition~\ref{thm:primeinJ},
 except for $25$ and $49$,
 exceptionals belong to the set $J$ with both of type ($\I$) and ($\II$).  
 In this section,
 we investigate the existence of infinitely many ordinaries and exceptionals of each type. 

 We first show the following assertion on ordinaries outside of $J$. 

\begin{thm}
 In each type of $(\I)$ and $(\II)$,
 there exists ordinary $m\notin J$ infinitely many.
\end{thm}
\begin{proof}
 It is easy to see that $p\mathbb{Z}\cap J=\emptyset$ if and only if
 $\bigl(\frac{c'}{p}\bigr)=-1$ for all $c\in\{\pm 1,\pm 3,\pm 5\}$
 where $c'=25-4c$ and $\bigl(\frac{\cdot}{p}\bigr)$ is the Legendre symbol.
 Since $\bigl(\frac{c'}{p}\bigr)=-1$ if and only if 
\begin{align*}
  p\equiv
\begin{cases}
 \pm 2 \pmod{5} & (c=\pm 5), \\
 \pm 2,\ \pm 5,\ \pm 6,\ \pm 8,\ \pm 13,\ \pm 14,\ \pm 15,\ \pm 17,\ \pm 18 \pmod{37} & (c=-3),\\
 \pm 2,\ \pm 3,\ \pm 8,\ \pm 10,\ \pm 11,\ \pm 11,\ \pm 12,\ \pm 14 \pmod{29} & (c=-1),\\
 \pm 2,\ \pm 8,\ \pm 10 \pmod{21} & (c=1),\\
 \pm 2,\ \pm 5,\ \pm 6 \pmod{13} & (c=3),
\end{cases}
\end{align*}
 this is equivalent to say that $p$ is of the form $p=at+b$ where 
 $a=5\cdot 13\cdot 21\cdot 29\cdot 37=1464645$, $b\in\{2,8,32,97,128,242,\ldots,1464637,1464643\}$ and $t\in\bZ$
 from the Chinese reminder theorem. 
 The Dirichlet theorem of arithmetic progression tells us
 there exists infinitely many primes of such forms
 (the first few are given by $97,577,827,853,947,\ldots$) and hence we have the assertion of type (I).
% $2\cdot 6\cdot 6\cdot 14\cdot 18=18144$ forms of such $p$;
 Moreover, for each prime $p$ satisfying the above condition,
 one can take a prime $q$ satisfying $p<q<2p$ because of the Bertrand-Chebyshev theorem, 
 and then $pq$ is not in $J$.
 This shows the assertion of type ($\II$). 
\end{proof}

 Next, we discuss about infinitely many existence of both ordinaries and exceptionals inside of $J$
 (we remark that, from Theorem~\ref{thm:exceptionalI}, there are no ordinaries in $J$ of type (I)).
 To state our results,
 we recall the well-known conjecture of Hardy-Littlewood \cite{{HardyLittlewood1923}} and Bateman-Horn \cite{BatemanHorn1962}.

\begin{conj}
\label{conj:HLC}
 Let $f_1(x),\ldots,f_r(x)\in\bZ[x]$ and $f(x)=f_1(x)\cdots f_r(x)$.
 Suppose that $f_1(x),\ldots,f_r(x)$ satisfy the following conditions:
\begin{itemize}
 \item[$\mathrm{(i)}$] $f_1(x),\ldots,f_r(x)$ are distinct.
 \item[$\mathrm{(ii)}$] $f_1(x),\ldots,f_r(x)$ are irreducible in $\bZ[x]$.
 \item[$\mathrm{(iii)}$] The leading coefficients of $f_1(x),\ldots,f_r(x)$ are positive.
 \item[$\mathrm{(iv)}$] There is no prime $\ell$ so that $\ell\,|\,f(n)$ for all $n\in\bZ_{>0}$.
\end{itemize}
 Then, we have 
\begin{align*}
  \pi(f_1,\ldots,f_r;x)
&=\#\bigl\{n\le x\,\bigl|\,f_1(n),\ldots,f_r(n) \ \text{are all prime}\bigr\}\\
&\sim \frac{1}{(\deg{f_1})\cdots (\deg{f_r})}C(f_1,\ldots,f_r)\frac{x}{(\log{x})^r},%\int^{x}_{2}\frac{1}{(\log{t})^r}dt
\end{align*}
 where $C(f_1,\ldots,f_r)$ is the Hardy-Littlewood constant defined by 
\[
 C(f_1,\ldots,f_r)=\prod_{p}\Bigl(1-\frac{\nu_{f}(p)}{p}\Bigr)\Bigl(1-\frac{1}{p}\Bigr)^{-r}
\]
 with $\nu_f(p)$ being the number of solutions $n$ in $\bZ_p$ of the congruence $f(n)\equiv 0 \! \pmod{p}$.
\end{conj}

 Now, we can state our results.

\begin{thm}
\label{thm:RamanujanHL}
 Under Conjecture~\ref{conj:HLC},
\begin{enumerate}
 \item[$(\mathrm{1})$] there exists exceptional $m$ infinitely many both of types $(\I)$ and $(\II)$.
 \item[$(\mathrm{2})$] there exists ordinary $m$ infinitely many of type $(\II)$.
\end{enumerate} 
\end{thm}

 To prove the assertion, we use the following lemma.

\begin{lem}
\label{lem:pq}
 For $a,y\in\mathbb{Z}_{>0}$ and $c\in\mathbb{Z}$, let
\begin{align*}
%\label{def:p}
 p&=p(a,y)=a^2(2a+1)^2y^2-a(2a+1)(8a+5)y+(4c-9)a^2+(4c-5)a+c,\\
%\label{def:q}
 q&=q(a,y)=16a^4y^2-8a^2(8a+1)y+4(4c-9)a^2+16a+1,\\
%\label{def:k}
 k&=k(a,y)=4a^3(2a+1)y^2-a(32a^2+20a+1)y+2(4c-9)a^2+(4c-1)a.
\end{align*}
\begin{enumerate}
 \item[$\mathrm{(1)}$] The identity $pq=k^2+5k+c$ holds.
 Moreover, $p<q<4p$ for $y\gg 0$ with  
\begin{align*}
 \lim_{y\to\infty}\frac{q}{p}=\Bigl(2-\frac{2}{2a+1}\Bigr)^2<4, \quad 
 \lim_{a\to\infty}\lim_{y\to\infty}\frac{q}{p}=4. 
\end{align*}
 \item[$\mathrm{(2)}$] If we consider $p$ and $q$ as polynomials in $\mathbb{Z}[y]$,
 then each of them satisfies the four conditions in Conjecture~\ref{conj:HLC} for any $c\in\{\pm 1,\pm 3,\pm 5\}$
 when $a\equiv 1,4,7,13\pmod{15}$.
\end{enumerate}
\end{lem}
\begin{proof}
 The identity $pq=k^2+5k+c$ and the above limit formulas for $\frac{q}{p}$ can be checked directly. 
 Moreover, since the coefficients of $y^2$ in both $q-p$ and $4p-q$ are positive for all $a>0$,
% from the fact that there exists finitely many $m\in J$ such that
% it is of the form of both $m=p(4p-1)$ and $p(4p-3)$ (see Remark~\ref{rem:finitem}),
 one sees that $p<q<4p$ for $y\gg 0$.
 Now, let us write $p=p(y)$ and $q=q(y)$.
 For all $a>0$, $p(y)$ and $q(y)$ satisfy the conditions $\mathrm{(i)}$ and $\mathrm{(iii)}$ obviously.
 Moreover, one sees that $(\mathrm{ii})$ is also true for all $a>0$ and $c\in\{\pm 1,\pm 3,\pm 5\}$
 since $p(y)$ and $q(y)$ have the non-square discriminants $c'a^2(2a+1)^4$ and $2^8c'a^6$, respectively.

 Put $d_p$ (resp. $d_q$) the greatest common divisor of the coefficients of $p(y)$ (resp. $q(y)$).
 Under the primitive situation of $p$ and $q$, that is, $d_p=d_q=1$, 
 it is sufficient to check the condition $(\mathrm{iv})$ only for the case $\ell=2,3$ because $\deg p=\deg q=2$.

 At first, it is easy to see that $d_p=((a,5)a,c)$ and $d_q=1$.
 Therefore, for all $a>0$ and $c\in\{\pm 1,\pm 3,\pm 5\}$ with $(a,c)=1$,
 the polynomials $p(y)$ and $q(y)$ are both primitive.
 When $\ell=2$, the condition is obvious for all $a>0$ and $c\in\{\pm 1,\pm 3,\pm 5\}$
 because $p(0)\equiv q(0)\equiv 1\!\pmod{2}$.
 The values of $p(y)$ at $y=0,1,2$ are congruent modulo $\ell=3$ to
 $ca^2+(c+1)a+c$, $(c+2)a^2+(c+2)a+c$, $(c+2)a^2+(c+2)a+c$, respectively. 
 Also, we have 
 $ca^2+a+1$, $(c+2)a^2+1$, $ca^2+2a+1$ for the values $q(y)$ at $y=0,1,2$, respectively. 
 Thus, except for the case $(a,c)\equiv (0,0), (2,0)\!\pmod{3}$, $p(y)q(y)$ is not congruent to the zero polynomial
 modulo $3$ for all $a>0$ and $c\in\{\pm 1,\pm 3,\pm 5\}$.

 Summing up the above discussion, 
 the polynomials $p(y)$ and $q(y)$ satisfy the condition $(\mathrm{iv})$
 if $a\not\equiv 0\!\pmod{5}$ when $c=\pm 5$, $a\equiv 1\!\pmod{3}$ when $c=\pm 3$ and for all $a$ when $c=\pm 1$.
 Therefore, solving the congruences $a\not\equiv 0\!\pmod{5}$ and $a\equiv 1\!\pmod{3}$,
 we obtain the second assertion.
\end{proof}

\begin{proof}
[Proof of Theorem~\ref{thm:RamanujanHL}]
 Since each of the six polynomial $f_c(x)$ with $c\in\{\pm 1,\pm 3,\pm 5\}$ satisfies
 the four conditions in Conjecture~\ref{conj:HLC},
 together with Theorem~\ref{thm:exceptionalI},
 one obtains the assertion for exceptionals of type $(\I)$.
 
 Now, let $\xi_1=x_1^2=2.0451\ldots,\xi_2=x_2^2=3.9365\ldots$ be the constants obtained in Theorem~\ref{thm:pqinJallc}.
 Take $a\in\mathbb{Z}_{>0}$ satisfying $a\equiv 1,4,7,13\!\pmod{15}$ and $(2-\frac{2}{2a+1})^2<\xi_1$, that is, $a=1$. 
 Then, under Conjecture~\ref{conj:HLC},
 the corresponding $p(y)$ and $q(y)$ in Lemma~\ref{lem:pq} represent infinitely many primes at the same time.
 Moreover, if both $p(y)$ and $q(y)$ are prime, then $m=p(y)q(y)=f_c(k(y))\in J$ and,
% since $\lim_{y\to\infty}\frac{q}{p}=\frac{16}{9}=1.7777\ldots<\xi_1$ 
% where $\xi_1$ is the constant obtained in
 form Theorem~\ref{thm:pqinJallc}, it is exceptional.
% Therefore, there are infinitely many exceptional $m=pq\in J$ and hence
 This shows the assertion for exceptionals of type $(\II)$.
 Furthermore, if we take $a\in\mathbb{Z}_{>0}$ satisfying $a\equiv 1,4,7,13\!\pmod{15}$ and $(2-\frac{2}{2a+1})^2>\xi_2$
 (notice that the smallest such $a$ is $64$),
 under Conjecture~\ref{conj:HLC}, from Theorem~\ref{thm:pqinJallc} again,
 one similarly proves the assertion for ordinaries of type $(\II)$.
 This completes the proof.
\end{proof}

\begin{example}
 Consider the case where $a=1$ and $c=-5$, that is, 
\begin{align*}
 p=9y^2-39y-59, \quad q=16y^2-72y-99.
\end{align*}
 Then, as we have seen above, $m=pq$ is exceptional if both $p$ and $q$ are prime for sufficiently large $y\gg 0$.
% we have $\hat{\mu}_{l_0+2}\le \mathrm{RB}_{l_0+2}$ with $m=pq$ for $y\gg 0$.
 Notice that, since $1<(2-\frac{2}{2a+1})=1.3333\ldots<\gamma_1=1.3843\ldots$
 where $\gamma_1$ is defined in Remark~\ref{rem:spectralordering},
 the inequality $\mu^{(1)}<\mu^{(2)}<\mu^{(0)}=\hat{\mu}<\RB$ holds for such $m$.
 The first few of such $p$ and $q$ are given in Table~4.

\begin{table}[htbp]
\begin{center}
{\footnotesize 
{\renewcommand\arraystretch{1.2}
\begin{tabular}{c||c|c|c||c|c|c}
  $y$ & $p$ & $q$ & $\frac{q}{p}$ & $\mu^{(0)}-\RB$ & $\mu^{(1)}-\RB$ & $\mu^{(2)}-\RB$ \\
\hline
\hline
 $7$ & $109$ & $181$ & $1.660\ldots$ &
 $-1.11\times 10^{-2}$ & $-8.21\times 10^{-2}$ & $-2.17\times 10^{-2}$ \\
\hline
 $17$ & $1879$ & $3301$ & $1.756\ldots$ & 
 $-7.58\times 10^{-4}$ & $-4.86\times 10^{-3}$ & $-1.09\times 10^{-3}$ \\
\hline
 $25$ & $4591$ & $8101$ & $1.764\ldots$ &
 $-3.11\times 10^{-4}$ & $-1.98\times 10^{-3}$ & $-4.42\times 10^{-4}$ \\
\hline
 $35$ & $9601$ & $16981$ & $1.768\ldots$ & 
 $-1.49\times 10^{-4}$ & $-9.50\times 10^{-4}$ & $-2.09\times 10^{-4}$ \\
\hline
 $40$ & $12781$ & $22621$ & $1.768\ldots$ &
 $-1.12\times 10^{-4}$ & $-7.13\times 10^{-4}$ & $-1.57\times 10^{-4}$ \\
\hline
 $62$ & $32119$ & $56941$ & $1.772\ldots$ &
 $-4.46\times 10^{-5}$ & $-2.83\times 10^{-4}$ & $-6.20\times 10^{-5}$ \\
\hline
 $82$ & $57259$ & $101581$ & $1.774\ldots$ &
 $-2.50\times 10^{-5}$ & $-1.59\times 10^{-4}$ & $-3.47\times 10^{-5}$ \\
\hline
 $104$ & $93229$ & $165469$ & $1.774\ldots$ &
 $-1.53\times 10^{-5}$ & $-9.77\times 10^{-5}$ & $-2.12\times 10^{-5}$ % \\
%\hline
% $134$ & $156319$ & $277549$ & $1.332$ &
% $-9.16\times 10^{-6}$ & $-5.83\times 10^{-5}$ & $-1.26\times 10^{-5}$ \\
%\hline
% $139$ & $168409$ & $299029$ & $1.332$ &
% $-8.50\times 10^{-6}$ & $-5.41\times 10^{-5}$ & $-1.17\times 10^{-5}$
\end{tabular}
}
}
 \caption{Differences between $\mu^{(i)}$ and $\RB$ for $m=pq$ with $a=1$, $c=-5$.}
\end{center}
\end{table}

 On the other hand, if we replace $c$ with $-7$, that is, 
\begin{align*}
 p=9y^2-39y-77, \quad q=16y^2-72y-131.
\end{align*}
 Then, $m=pq\notin J$ and hence $m$ is ordinary from Theorem~\ref{thm:criterionSl0+2}.
 Actually, as one finds from Table~5, the inequality $\RB<\mu^{(0)}=\hat{\mu}$ holds.

\begin{table}[htbp]
\begin{center}
{\footnotesize
{\renewcommand\arraystretch{1.2}
\begin{tabular}{c||c|c|c||c|c|c}
  $y$ & $p$ & $q$ & $\frac{q}{p}$ & $\mu^{(0)}-\RB$ & $\mu^{(1)}-\RB$ & $\mu^{(2)}-\RB$ \\
\hline
\hline
 $13$ & $937$ & $1637$ & $1.747\ldots$ &
 $1.07\times 10^{-4}$ & $-8.13\times 10^{-3}$ & $-6.21\times 10^{-4}$ \\
\hline
 $43$ & $14887$ & $26357$ & $1.770\ldots$ &
 $4.70\times 10^{-6}$ & $-5.11\times 10^{-4}$ & $-3.36\times 10^{-5}$ \\
\hline
 $60$ & $29983$ & $53149$ & $1.772\ldots$ &
 $2.30\times 10^{-6}$ & $-2.54\times 10^{-4}$ & $-1.64\times 10^{-5}$ \\
\hline
 $81$ & $55813$ & $99013$ & $1.774\ldots$ &
 $1.22\times 10^{-6}$ & $-1.36\times 10^{-4}$ & $-8.73\times 10^{-6}$ \\
\hline
 $158$ & $218437$ & $387917$ & $1.775\ldots$ &
 $3.11\times 10^{-7}$ & $-3.48\times 10^{-5}$ & $-2.19\times 10^{-6}$ \\
\hline
 $211$ & $392383$ & $697013$ & $1.776\ldots$ &
 $1.73\times 10^{-7}$ & $-1.93\times 10^{-5}$ & $-1.21\times 10^{-6}$ \\
\hline
 $225$ & $446773$ & $793669$ & $1.776\ldots$ &
 $1.52\times 10^{-7}$ & $-1.70\times 10^{-5}$ & $-1.06\times 10^{-6}$ \\
\hline
 $249$ & $548221$ & $973957$ & $1.776\ldots$ &
 $1.23\times 10^{-7}$ & $-1.38\times 10^{-5}$ & $-8.69\times 10^{-7}$ % \\
%\hline
% $270$ & $645493$ & $1146829$ & $1.332$ &
% $1.05\times 10^{-7}$ & $-1.17\times 10^{-5}$ & $-7.38\times 10^{-7}$, \\
%\hline
% $288$ & $735187$ & $1306237$ & $1.332$ &
% $9.24\times 10^{-8}$ & $-1.03\times 10^{-5}$ & $-6.47\times 10^{-7}$ 
\end{tabular}
}
}
\end{center}
 \caption{Differences between $\mu^{(i)}$ and $\RB$ for $m=pq$ with $a=1$, $c=-7$.}
\end{table}
\end{example}

\begin{example}
 Consider the case where $a=64$ and $c=5$, that is,  
\begin{align*}
 p=68161536y^2-4268352y+46021, \quad q=268435456y^2-16809984y+181249.
\end{align*}
 In this case, $m=pq$ is ordinary if both $p$ and $q$ are prime for sufficiently large $y\gg 0$.
 Notice that, since $\gamma_5(5)=1.9839\ldots<(2-\frac{2}{2a+1})=1.9845\ldots<2$ where $\gamma_5(5)$ is also defined in
 Remark~\ref{rem:spectralordering}, the inequality $\mu^{(2)}<\mu^{(0)}<\RB<\mu^{(1)}=\hat{\mu}$ holds for such $m$.
 See Table~6.

\begin{table}[htbp]
\begin{center}
{\footnotesize  
{\renewcommand\arraystretch{1.2}
\begin{tabular}{c||c|c|c||c|c|c}
  $y$ & $p$ & $q$ & $\frac{q}{p}$ & $\mu^{(0)}-\RB$ & $\mu^{(1)}-\RB$ & $\mu^{(2)}-\RB$ \\
\hline
\hline
$39$  & $103507276549$  & $407634920449$ & $3.938\cdots$ & $-5.79\times 10^{-11}$ & $2.17\times 10^{-13}$ & $-6.61\times 10^{-11}$ \\
\hline
$134$ & $1223336627269$ & $4817774691329$ & $3.938\cdots$ & $-4.90\times 10^{-12}$ & $1.84\times 10^{-14}$ & $-5.59\times 10^{-12}$ \\
\hline
$165$ & $1854993585541$ & $7305381823489$ & $3.938\cdots$ & $-3.23\times 10^{-12}$ & $1.21\times 10^{-14}$ & $-3.69\times 10^{-12}$ \\
\hline
$178$ & $2158870385989$ & $8502116992001$ & $3.938\cdots$ & $-2.77\times 10^{-12}$ & $1.04\times 10^{-14}$ & $-3.17\times 10^{-12}$ \\
\hline
$279$ & $5304571299589$ & $20890594526209$ & $3.938\cdots$ & $-1.13\times 10^{-12}$ & $4.25\times 10^{-15}$ & $-1.29\times 10^{-12}$ \\
\hline
$433$ & $12777690072709$ & $50321416668161$ & $3.938\cdots$ & $-4.69\times 10^{-13}$ & $1.76\times 10^{-15}$ & $-5.35\times 10^{-13}$ \\
\hline
$468$ & $14927014718149$ & $58785940423681$ & $3.938\cdots$ & $-4.02\times 10^{-13}$ & $1.51\times 10^{-15}$ & $-4.58\times 10^{-13}$ \\
\hline
$499$ & $16970160763909$ & $66832308978689$ & $3.938\cdots$ & $-3.53\times 10^{-13}$ & $1.32\times 10^{-15}$ & $-4.03\times 10^{-13}$ 
%\hline
% $611$ & $169695342289321$ & $672238936050421$ & $1.990$ &
% $-5.65\times 10^{-15}$ & $3.10\times 10^{-14 }$ & $-1.05\times 10^{-14 }$  \\
%\hline
% $874$ & $347231154423409$ & $1375537470211501$ & $1.990$ &
% $-2.76\times 10^{-15}$ & $1.51\times 10^{-14}$ & $-5.18\times 10^{-15}$ 
\end{tabular}
}
}
\end{center}
 \caption{Differences between $\mu^{(i)}$ and $\RB$ for $m=pq$ with $a=64$, $c=5$.}
\end{table}
\end{example}

\begin{remark}
% We here remark that, even if we can prove the existence of exceptional $m\in J$,
% one can not easily conclude that there exists infinitely many such $m$'s.
 Let us denote the fractional part of a real number $x$ by $\{x\}$.
% see this in the case of $(\mathrm{i})$, that is, $m=p$.
% In this case, we have $l_0=2\Gauss{\sqrt{p}-\frac{3}{2}}+1$.
 If the sequence  $\{\{\sqrt{p}-\frac{3}{2}\}\}_{p\,:\,\text{prime}}$ is included in a closed interval, 
 then one easily sees that there can not be infinitely many exceptional primes.
 In this sense, this phenomena on the existence of exceptional primes is also related to $\{\sqrt{p}\}$
 which distributes uniformly in the interval $[0,1)$ ({\it cf}. \cite{{Dieter1975,{Dieter1976}}}).
% Remark that, since $l_0(p)=2\Gauss{\sqrt{p}-\frac{3}{2}}+1$,
% this phenomena is deeply related to the distribution of the fractional part of $\sqrt{p}$.
% Actually, in \cite{{Dieter1975,{Dieter1976}}},
% it was shown that 
% the sequence $\{\{\sqrt{p}\}\}_{p\,:\,\text{prime}}$ where $\{x\}$ denotes the
% fractional part of $x$ is uniformly distributed modulo one.
% Hence, one can not easily conclude that there exists only finitely many exceptional $p$.
% In fact, we expect that the following holds;
\end{remark}

\section{Numerical consideration}
 
 Let $\rho_{E}(x)$ be the number of exceptionals $m\le x$.
 It is now natural to ask how $\rho_{E}(x)$ behaves as $x$ tends to infinity.
 The aim of this section is to consider this question
 by giving some conjectures which are obtained by numerical studies.
 Notice that to investigate $\rho_E(x)$ it is enough to know $\pi_{E}(c;x)$ for $c\in\{\pm 1,\pm 3,\pm 5\}$
 where $\pi_{E}(c;x)$ is the number of $k\le x$ such that $f_c(k)=k^2+5k+c$ is exceptional, 
 since $\rho_E(x)\sim \sum_{c\in \{\pm 1,\pm 3,\pm 5\}}\pi_E(c;\sqrt{x})$ from Theorem~\ref{thm:criterionSl0+2}.
 Moreover,
 it is sufficient to investigate $\pi^{(\I)}_{E}(c;x)$ and $\pi^{(\II)}_{E}(c;x)$,
 the number of $k\le x$ such that $f_c(k)$ is exceptional of type $(\I)$ and $(\II)$, respectively,
 because of the identity
\[
 \pi_{E}(c;x)\sim \pi^{(\I)}_{E}(c;x)+\pi^{(\II)}_{E}(c;x),
\]
 which is immediate from Proposition~\ref{thm:primeinJ}. 

% where $\pi^{(\I)}_{E}(x;c)$ and $\pi^{(\II)}_{E}(x;c)$ are respectively 
% the number of $k\le x$ such that $f_c(k)$ is exceptional of type $(\I)$ and $(\II)$, 
% we  $\pi^{(\I)}_{E}(x;c)$ and $\pi^{(\II)}_{E}(x;c)$. 
%In this section, we numerically and heuristically study the asymptotic distribution of
% these functions as $x\to\infty$ and give several conjectures.

\subsection{Distribution of exceptionals of type ($\boldsymbol{\I}$)}

 From Theorem~\ref{thm:exceptionalI},
 we have $\pi^{(\I)}_{E}(c;x)=\pi(f_c;x)$,
 where $\pi(f;x)$ with $f\in\mathbb{Z}[x]$ is defined in Conjecture~\ref{conj:HLC}.
 Hence, from the conjecture of Hardy-Littlewood and Bateman-Horn,
 the asymptotic behavior of $\pi^{(\I)}_{E}(c;x)$ is expected as follows.

\begin{conj}
\label{conj:EXPECTEDexceptionalp}
 It holds that 
\begin{equation*}
%\label{for:EXPECTEDexceptionalp}
 \pi^{(\I)}_{E}(c;x)\sim C^{(\I)}(c)\frac{x}{\log{x}},
\end{equation*}
 where $C^{(\I)}(c)$  
\begin{equation*}
%\label{for:HLconstant}
 C^{(\I)}(c)=\frac{C(f_c)}{2}
=\prod_{p\ge 3}\Bigl(1-\frac{\bigl(\frac{c'}{p}\bigl)}{p-1}\Bigr)
=
\begin{cases}
 1.18219\ldots & (c=-5),\\
 1.18219\ldots & (c=-3),\\
 1.12674\ldots & (c=-1),\\
 0.927881\ldots & (c=1),\\
 0.807233\ldots & (c=3),\\
 1.77328\ldots & (c=5),
\end{cases}
\end{equation*}
 with $c'=25-4c$.
\end{conj}

\subsection{Distribution of exceptionals of type $(\boldsymbol{\II})$}
\label{sec:etype2}

 For $a>1$, let $P_2(a)$ be the set of all $pq$ where $p$ and $q$ are distinct primes satisfying $p<q<ap$. 
 Moreover, for a polynomial $f\in\mathbb{Z}[x]$, let $\pi_2(f,a;x)$ be the number of $k\le x$ such that $f(k)\in P_2(a)$.
 From Theorem~\ref{thm:pqinJallc} and the observation in Remark~\ref{rem:spectralordering},
 one may expect that $\pi^{(\II)}_E(c;x)$ asymptotically behaves as $\pi_2\bigl(f_c,\gamma_5(c)^2;x\bigr)$,
 where $\gamma_5(c)$ is a constant also defined in Remark~\ref{rem:spectralordering}.
 We here notice that Conjecture~\ref{conj:HLC} with $r=1$
 asserts that $\pi(f;x)\asymp \pi(x)$ for any $f\in\mathbb{Z}[x]$ satisfying the conditions in Conjecture~\ref{conj:HLC},
 where $\pi(x)\sim\frac{x}{\log{x}}$ is the number of primes $p\le x$. 
 Based on this observation, we may expect the same situation for $\pi_2(f,a;x)$, that is, $\pi_2(f,a;x)\asymp \pi_2(a;x)$
 where $\pi_2(a;x)$ is the number of $m\le x$ such that $m\in P_2(a)$.
 For $\pi_2(a;x)$, we can say the following (for more precise discussion, see \cite{DeckerMoree2008,Hashimoto2009}).

\begin{lem}
 It holds that   
\begin{equation*}
%\label{for:EXPECTEDpi^2}
 \pi_2(a;x)\asymp \frac{x}{(\log{x})^2}.
\end{equation*}
\end{lem}
\begin{proof}
 Fix any prime number $p_0$.
 For $x\ge ap_0^2$, by the prime number theorem, we have
\begin{align*}
 \pi^2(x;a)
&=\sum_{p\le \frac{\sqrt{x}}{\sqrt{a}}}\sum_{p<q<ap}1
+\sum_{\frac{\sqrt{x}}{\sqrt{a}}<p\le \sqrt{x}}\sum_{p<q\le \frac{x}{p}}1\\
%&=\sum_{p\le \frac{\sqrt{x}}{\sqrt{a}}}\Bigl(\pi(ap)-\pi(p)\Bigr)
%+\sum_{\frac{\sqrt{x}}{\sqrt{a}}<p\le \sqrt{x}}\Bigl(\pi\bigl(\frac{x}{p}\bigr)-\pi(p)\Bigr)\\
&=\sum_{p\le \frac{\sqrt{x}}{\sqrt{a}}}\Bigl(\frac{ap}{\log{ap}}-\frac{p}{\log{p}}+O(1)\Bigr)
+\sum_{\frac{\sqrt{x}}{\sqrt{a}}<p\le
 \sqrt{x}}\Bigl(\frac{\frac{x}{p}}{\log{\frac{x}{p}}}-\frac{p}{\log{p}}+O(1)\Bigr)\\
&=\sum_{p_0\le p\le \frac{\sqrt{x}}{\sqrt{a}}}\Bigl(\frac{ap}{\log{ap}}-\frac{p}{\log{p}}\Bigr)
+\sum_{\frac{\sqrt{x}}{\sqrt{a}}<p\le \sqrt{x}}\Bigl(\frac{\frac{x}{p}}{\log{\frac{x}{p}}}-\frac{p}{\log{p}}\Bigr)+O\Bigl(\frac{\sqrt{x}}{\log{x}}\Bigr).
\end{align*}
 Let us write the first and the second sums of the rightmost hand side as $A$ and $B$, respectively.

 Since $x\ge ap_0^2$, we have $\frac{a}{c(a)+1}\frac{p}{\log{p}}<\frac{ap}{\log{ap}}<a\frac{p}{\log{p}}$
 where $c(a)=\frac{\log{a}}{\log{p_0}}$.
 This shows that
\begin{equation}
\label{for:AAA}
 \frac{2}{a}\Bigl(\frac{a}{c(a)+1}-1\Bigr)\frac{x}{(\log{x})^2}\ll A\ll \frac{2}{a}(a-1)\frac{x}{(\log{x})^2}.
\end{equation}
 Here, we have used the formula 
\begin{equation}
\label{for:plogp}
 \sum_{p\le x}\frac{p}{\log{p}}\sim \frac{1}{2}\Bigl(\frac{x}{\log{x}}\Bigr)^2,
\end{equation}
 which follows from the Abel summation formula with the fact that
 there exists a constant $c>0$ such that 
 $\vartheta(x)=\sum_{p\le x}\log{p}=x+O(x\exp(-c\sqrt{\log{x}}))$ 
 (see, e.g., \cite{MontgomeryVaughan2007}).

 On the other hand, we have  
\begin{align*}
 \log{a}\frac{x}{(\log{x})^2}\sim \frac{x}{\log{x}}\sum_{\frac{\sqrt{x}}{\sqrt{a}}<p\le \sqrt{x}}\frac{1}{p}< \sum_{\frac{\sqrt{x}}{\sqrt{a}}<p\le \sqrt{x}}\frac{\frac{x}{p}}{\log{\frac{x}{p}}}
< \frac{2x}{\log{x}}\sum_{\frac{\sqrt{x}}{\sqrt{a}}<p\le \sqrt{x}}\frac{1}{p}\sim 2\log{a}\frac{x}{(\log{x})^2},
\end{align*}
 where we have used the fact that there exists constants $b$ and $c'>0$ such that 
 $\sum_{p\le x}\frac{1}{p}=\log{\log{x}}+b+O(\exp(-c'\sqrt{\log{x}}))$ (see \cite{MontgomeryVaughan2007} again).
 This together with \eqref{for:plogp} implies
\begin{equation}
\label{for:BBB}
 \Bigl(\log{a}-2+\frac{2}{a}\Bigr)\frac{x}{(\log{x})^2}\ll B\ll \Bigl(2\log{a}-2+\frac{2}{a}\Bigr)\frac{x}{(\log{x})^2}.
\end{equation}

 Combining \eqref{for:AAA} and \eqref{for:BBB}, we have 
\[
 \Bigl(\frac{2\log{p_0}}{\log{p_0}+\log{a}}+\log{a-2}\Bigr)\frac{x}{(\log{x})^2}\ll A+B\ll 2\log{a}\frac{x}{(\log{x})^2}.
\]
 Notice that if we take $p_0\ge 11$, then the coefficients of the leftmost hand side is positive for all $a>1$.
 This completes the proof.
\end{proof}

  These observations lead us to expect the following.

\begin{conj}
\label{conj:EXPECTEDexceptionalpq}
 There exists a constant $C^{(\II)}(c)$ such that  
\begin{equation*}
%\label{for:EXPECTEDexceptionalpq}
 \pi^{(\II)}_{E}(c;x)\sim C^{(\II)}(c)\frac{x}{(\log{x})^2}.
\end{equation*}
\end{conj}

 Here we give a numerical computation for the values
 $\pi^{(\II)}_{E}(c;x)/\frac{x}{(\log{x})^2}$ with $x\le 5\times 10^{7}$ for each $c\in\{\pm 1,\pm 3,\pm 5\}$ in Figure~7-12.
% We also give values of the above function at $x=10^{7}$,
% which expected constant $C^{(\II)}(c)$ are given as follows;  
%\begin{equation*}
%%\label{for:HLconstant2}
% C^{(\II)}(c)
%\fallingdotseq
%\begin{cases}
% 3.2813415594362142417\ldots & (c=-5),\\
% 1.7865965119822372128\ldots & (c=-3),\\
% 3.3300267690254931793\ldots & (c=-1),\\
% 2.8037380946073433835\ldots & (c=1),\\
% 2.8037380946073433835\ldots & (c=3),\\
% 5.6032155838931063263\ldots & (c=5).
%\end{cases}
%\end{equation*}

\begin{figure}[htbp]
\begin{center}
\begin{tabular}{ccc}
  \begin{minipage}{0.33\textwidth}
    \begin{center}
     \includegraphics[clip,width=50mm]{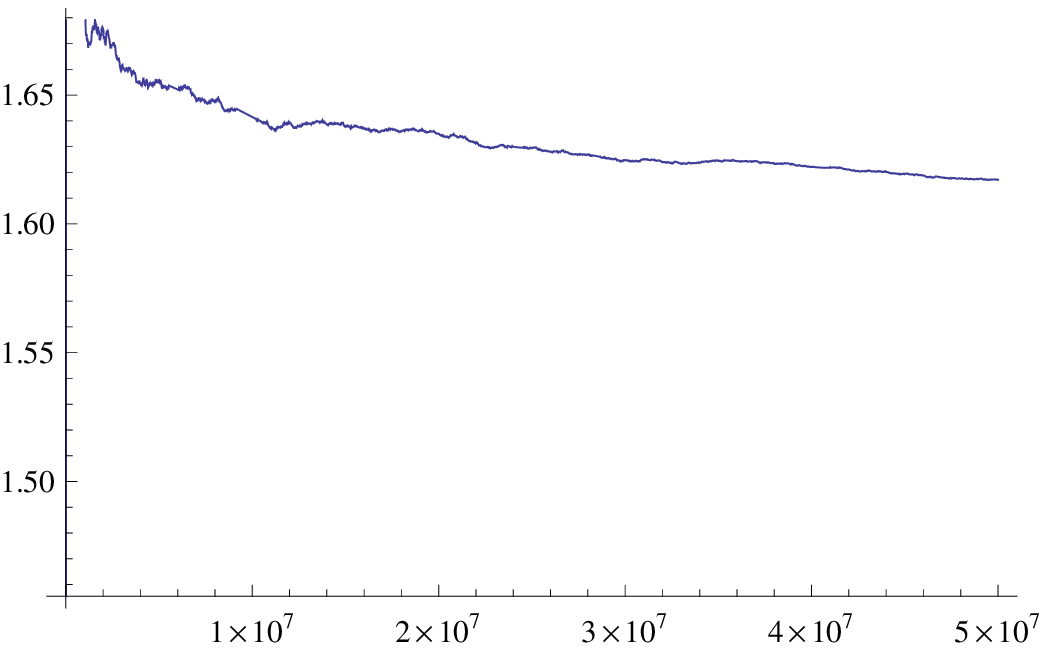}
     \caption{$c=-5$.}
    \end{center}
   \end{minipage}
   \begin{minipage}{0.33\textwidth}
    \begin{center}
     \includegraphics[clip,width=50mm]{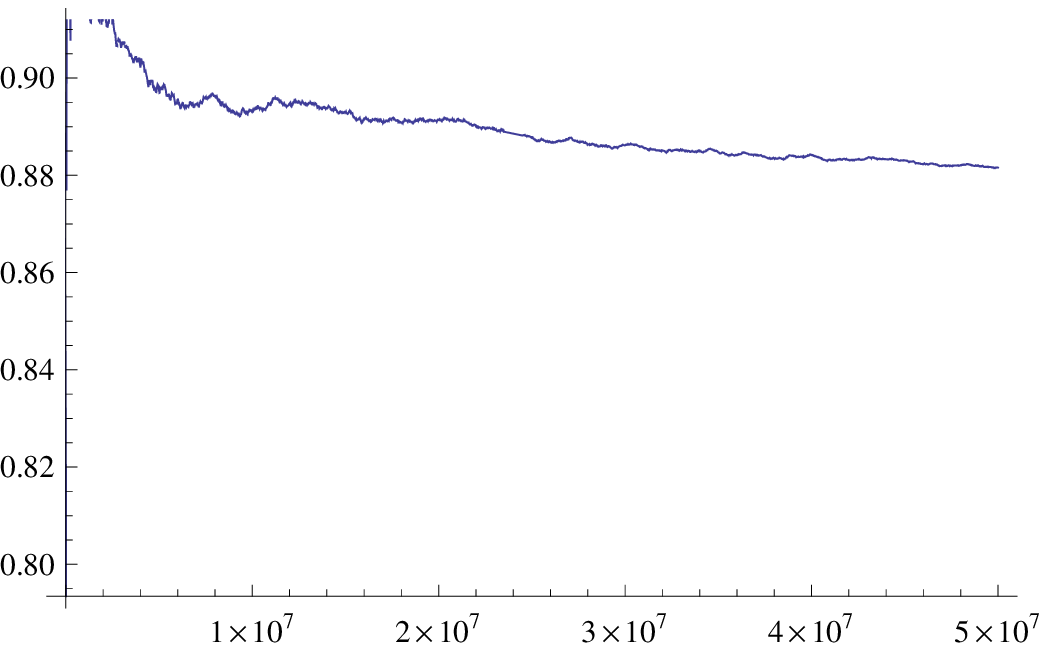}
     \caption{$c=-3$.}
    \end{center}
   \end{minipage}
      \begin{minipage}{0.33\textwidth}
    \begin{center}
     \includegraphics[clip,width=50mm]{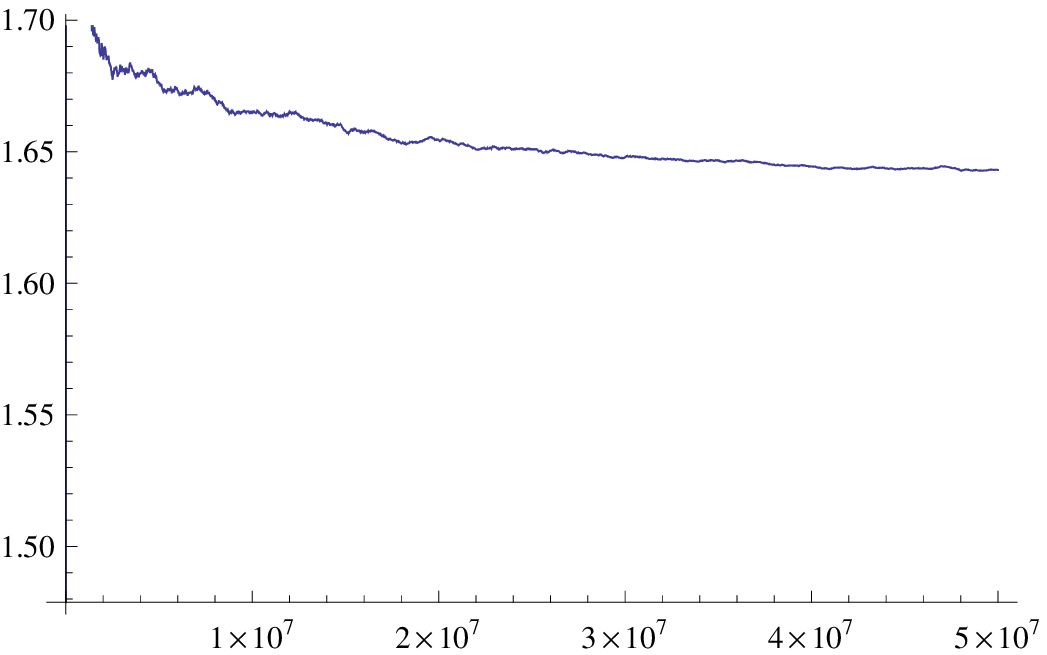}
     \caption{$c=-1$.}
    \end{center}
   \end{minipage}
  \end{tabular}
\end{center}

\begin{center}
\begin{tabular}{ccc}
  \begin{minipage}{0.33\textwidth}
    \begin{center}
     \includegraphics[clip,width=50mm]{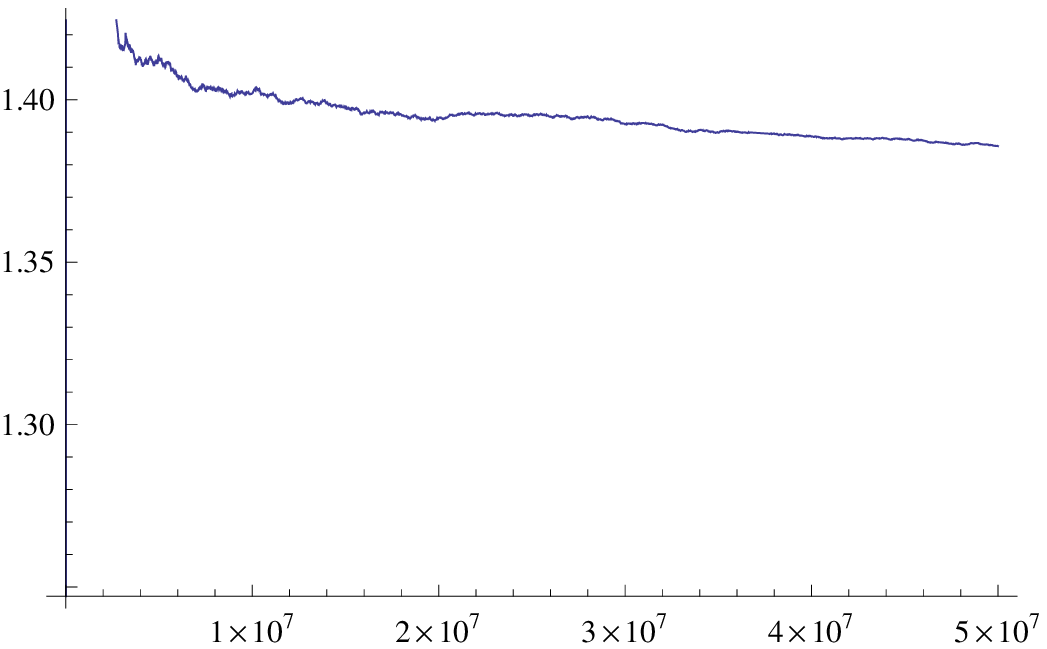}
     \caption{$c=1$.}
    \end{center}
   \end{minipage}
   \begin{minipage}{0.33\textwidth}
    \begin{center}
     \includegraphics[clip,width=50mm]{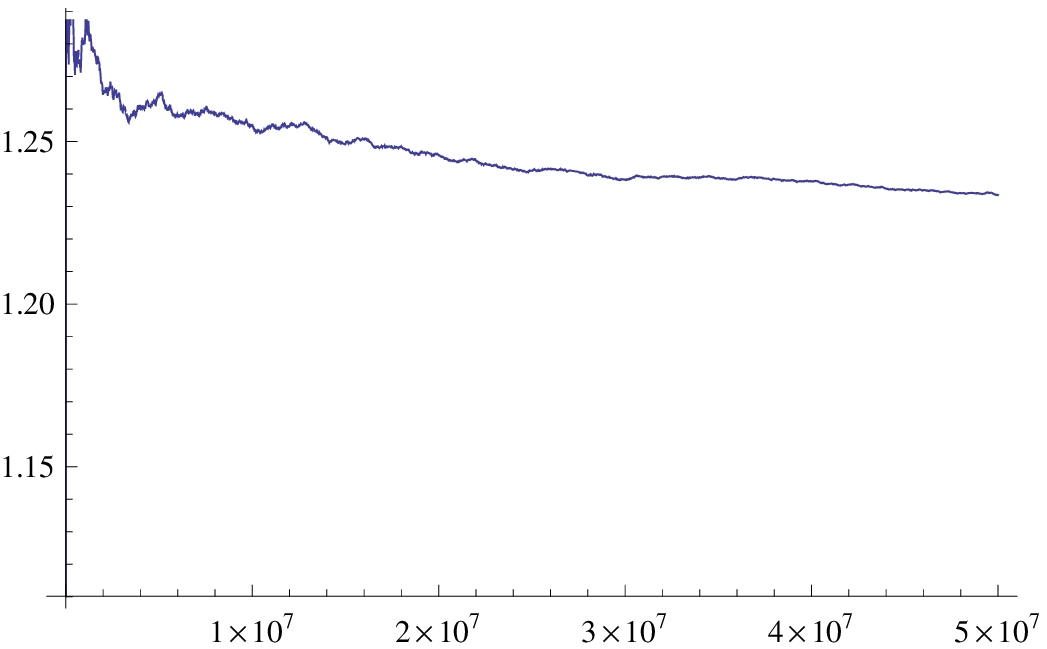}
     \caption{$c=3$.}
    \end{center}
   \end{minipage}
      \begin{minipage}{0.33\textwidth}
    \begin{center}
     \includegraphics[clip,width=50mm]{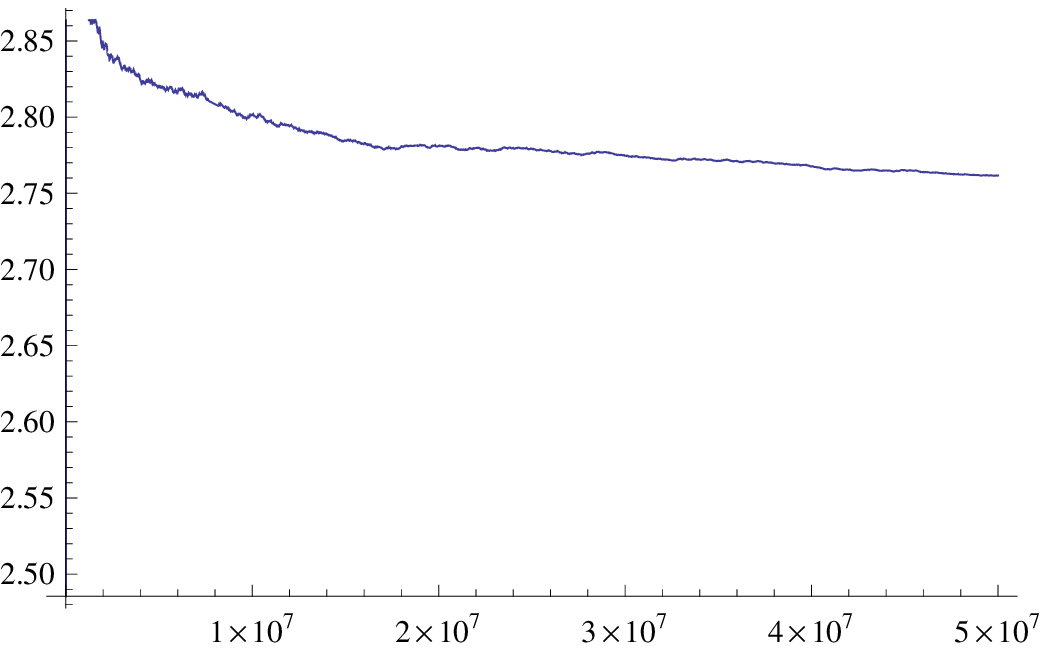}
     \caption{$c=5$.}
    \end{center}
   \end{minipage}
  \end{tabular}
\end{center}
\end{figure}

\begin{remark}
 Under Conjecture~\ref{conj:HLC}, 
 one can show the relation $\pi^{(\II)}_E(c;x)\gg \frac{x}{(\log{x})^2}$.
 Actually, from our construction of exceptionals in the proof of Theorem~\ref{thm:RamanujanHL},
 we have 
\begin{align*}
 \pi^{(\II)}_E(x;c)
\gg \#\bigl\{y\le x\,\bigl|\,\text{$p(1,y)$ and $q(1,y)$ are both primes}\bigr\}\asymp \frac{x}{(\log{x})^2}.
\end{align*} 
 Here, 
\begin{align*}
 p(1,y)=9y^2-39y+9c-14,\quad q(1,y)=16y^2-72y+16c-19.
\end{align*}
\end{remark}

\subsection{Some Remarks}

\begin{remark}
 Let $P_2$ be the set of all $pq$ where $p$ and $q$ are distinct primes with $p<q$.
 Moreover, for $f\in\mathbb{Z}[x]$,
 let $\pi_2(f;x)$ be the number of $k\le x$ such that $f(k)\in P_2$.
 Similar to the discussion in \S~\ref{sec:etype2},
 one may expect that $\pi_2(f;x)$ is asymptotically equal to a constant multiple of $\pi_2(x)$
 if $f$ satisfies suitable conditions, that is,   
\begin{equation}
\label{for:EXPECTEDexceptionalpq}
 \pi_{2}(f;x)\asymp \pi_2(x)\sim \frac{x\log\log{x}}{\log{x}}.
\end{equation}
 Here, $\pi_2(x)$ is the number of $m\le x$ such that $m\in P_2$.
 Notice that the second equality relation in \eqref{for:EXPECTEDexceptionalpq}
 was obtained by Landau \cite{Landau1900} (see also \cite{HardyWright1979}).

 A positive integer having at most two distinct prime factors
 is called an {\it almost prime}.
 When $f(x)$ is a quadratic polynomial,
 it is shown by Iwaniec \cite{Iwaniec1978} and Lemke-Oliver \cite{LemkeOliver2012}
 that there are infinitely many $k$ such that $f(k)$ is almost prime.
 More precisely, they prove that
\begin{equation}
\label{for:IL}
  \pi(f;x)+\pi_{2}(f;x)\gg \frac{x}{\log{x}}
\end{equation}
 if $f$ satisfies suitable conditions.
 Of course, the expectation \eqref{for:EXPECTEDexceptionalpq} is more stronger than the result \eqref{for:IL}.
 In Figure~13,
 we give a numerical computation of $\pi_{2}(f;x)/(\frac{x\log\log{x}}{\log{x}})$ for $x\le 5\times 10^{4}$
 with $f(k)=k^2+1$, which is studied in \cite{Iwaniec1978}.

\begin{figure}[htbp]
\begin{center}
 \includegraphics[clip,width=70mm]{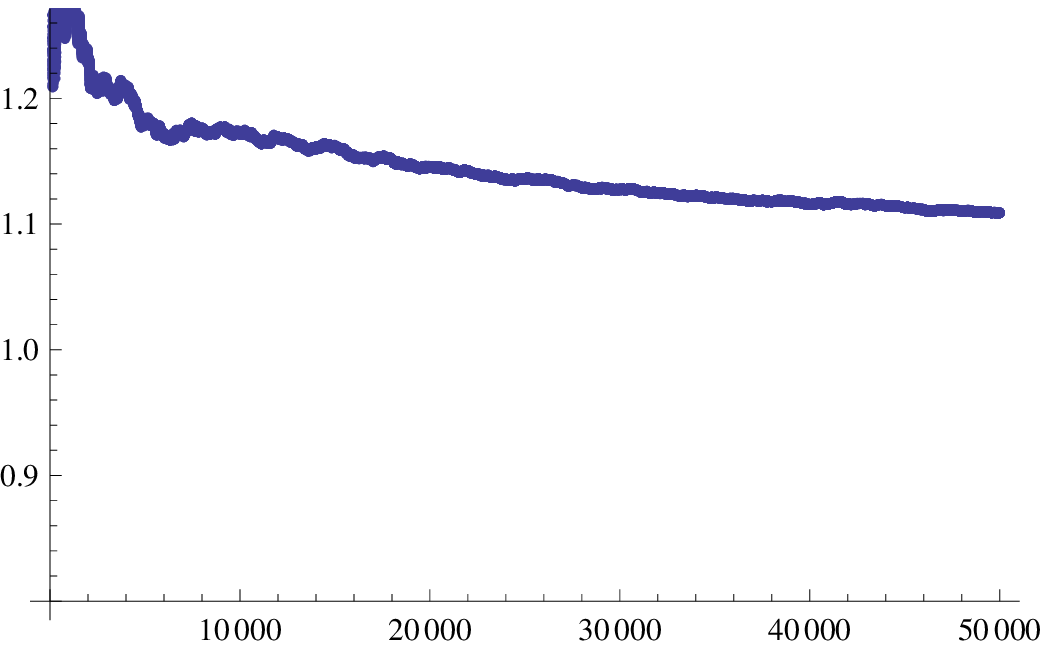}
     \caption{The asymptotic of $\pi_2(f;x)/(\frac{x\log{\log{x}}}{\log{x}})$ with $f(k)=k^2+1$ for $x\le 5\times 10^{4}$.}
\end{center}
\end{figure}

%\begin{figure}[htbp]
%\begin{center}
%\begin{tabular}{cc}
%  \begin{minipage}{0.5\textwidth}
%    \begin{center}
%     \includegraphics[clip,width=60mm]{Iwaniec1.eps}
%    \end{center}
%   \end{minipage}
%   \begin{minipage}{0.5\textwidth}
%    \begin{center}
%     \includegraphics[clip,width=60mm]{Iwaniec2.eps}
%     \end{center}
%   \end{minipage}
%  \end{tabular}
%\end{center}
%\caption{The asymptotic of $\pi_2(f;x)/(\frac{x\log{\log{x}}}{\log{x}})$ with $f(k)=k^2+1$ for $x\le 5\times 10^{4}$.}
%\end{figure}
\end{remark}

\begin{remark}
% From our consideration,
% exceptional $m$'s are numbers
% which are represented by the quadratic forms $f_c(x):=x^2+5x+c$ for some $c\in\{\pm 1,\pm 3,\pm 5\}$
% and have at most two  distinct primes factors satisfying some condition
% (except for the finite cases in Theorem~\ref{for:primeprimeinJ}, that is, the case $m=p^2\in J$).
% Therefore,
 If one can prove that
 there exists infinitely many exceptionals in the framework of graph theory,  
 then, from Theorem~\ref{thm:criterionSl0+2} and Proposition~\ref{thm:primeinJ},
 one may obtain a theorem of Iwaniec \cite{Iwaniec1978} and Lemke-Oliver \cite{LemkeOliver2012} type,
% which asserts that there are infinitely many $2$-almost primes expressed by a quadratic polynomial,
 for at least one of $f_c$. % may be obtained in such a framework.
 Much more stronger,
 if one can prove the existence of infinitely many exceptional primes in such a framework,
 then we can say that
 the conjecture of Hardy-Littlewood and Bateman-Horn is true for at least one of $f_c$. 
% Theorem~\ref{thm:immediately},\ref{thm:primeinJ} and \ref{thm:pqinJ},
% we know that $m$ has at most two distinct prime factors if $\hat{l}=l_0+2$
% (except for the finite cases in Theorem~\ref{for:exceptionalinJ}).
% Moreover, if it has exactly two distinct factors, then we need an extra condition $p<q<4p$.
% This implies that if we can prove that there exists infinitely many $m\in J$ such that $\hat{l}=l_0+2$ in the frame work
% on the graph theory, 
% then we can obtain a refinement of the Iwaniec \cite{Iwaniec1978} and Lemke-Oliver \cite{LemkeOliver2012} type of theorem,
% which asserts that there are infinitely many integer which is expressed as a quadratic polynomial having at most two prime factors,
% for one of the polynomial $f_c(x)$.
% Much more stronger, if we can prove the infinitely many existence of primes $p\in J$ such that $\hat{l}=l_0+2$,
% then this immediately implies that we prove the Hardy-Littlewood conjecture for one of $f_c(x)$. 
\end{remark}

\section{Ramanujan abelian graphs of odd order}

 Our problem can be discussed more general situation.
 Namely, we can determine $\hat{l}$ for any finite abelian group $G$ of odd order $m$, instead of $\mathbb{Z}_m$.
 Let $\widehat{G}$ be the dual group of $G$ and
 $\cS$ the set of all Cayley subset of $G$.
 Notice that, since $m$ is odd, there is no element in $G$ whose order is two.
 This means that $|S|$ is even for any $S\in\cS$ and hence $l(S)=m-|S|$ is always odd.
 We denote by $X(S)$ the Cayley graph of $G$ attached to $S\in\cS$
 and $\Lambda(S)$ the set of all eigenvalues of $X(S)$.
 As we have explained in Section~\ref{sec:CayleyGraph},
 it can be written as $\Lambda(S)=\{\lambda_{\chi}\,|\,\chi\in \widehat{G}\}$
 where $\lambda_{{\bf 1}_G}=|S|$ with ${\bf 1}_G$ being the trivial character of $G$ and
\[
 \lambda_{\chi}=\sum_{a\in S}\chi(a)=-\sum_{b\in G\setminus S}\chi(b), \qquad \chi\ne {\bf 1}_G.
\]
 From the same discussion as in the proof of Lemma~\ref{lem:trivial},
 it is immediate to see that $\hat{l}\ge l_0=2\Gauss{\sqrt{m}-\frac{3}{2}}+1$.
 Let us also call $G$ {\it ordinary} (resp. {\it exceptional}) if $\hat{l}=l_0$ (resp. $\hat{l}\ge l_0+2$).

 From the fundamental theorem of finite abelian groups,
 we may assume that $G$ is a direct sum of finite number of cyclic groups.
 We remark that we here do not consider $G=\mathbb{Z}_{3}\oplus\mathbb{Z}_{3}$
 because it can be checked that all the Cayley graphs of $G$ are Ramanujan.
 The following theorem says that there are only finitely many exceptionals $G$ which are not cyclic.

\begin{thm}
\label{thm:finiteabel}
 Let $G$ be a finite abelian group of odd order which is not cyclic
 Then, $G$ is ordinary except for the cases $G=\mathbb{Z}_{p}\oplus \mathbb{Z}_{p}$ where $p$ is odd prime with $3\le p\le 17$.
 In the exceptional cases, we have
\[
 \hat{l}=
\begin{cases}
 l_0+2 & p=7,11,13,17,\\
 l_0+4 & p=5.
\end{cases}
\]
\end{thm}

 Before giving a proof of the theorem, 
 it is convenient to prove the following lemma. 
 
\begin{lem}
\label{lem:keyabel}
 Let $G_1,G_2$ be finite abelian groups
 with $|G_1|=n_1,|G_2|=n_2$, respectively,
 where $n_1,n_2$ are odd integer with $n_1\le n_2$.
 If $n_2\ge 4n_1-3$, then $G=G_1\oplus G_2$ is ordinary.
\end{lem}
\begin{proof}
 Take a non-trivial character $\chi_0=\eta\otimes {\bf 1}_{G_2}\in\widehat{G}$ with $\eta$ being a non-trivial character of $G_1$.
 As in the proof of Proposition~\ref{thm:primeinJ}, 
 the condition $n_2\ge 4n_1-3$ implies that 
 one can take $S\in\cS_{l_0+2}$ as $G\setminus S\subset \{(0,g_2)\in G\,|\,g_2\in G_2\}$.
 This shows that 
\[
 \bigl|\lambda_{\chi_0}\bigr|
=\left|\sum_{(g_1,g_2)\in G\setminus S}\chi_0(g_1,g_2)\right|=l_0+2\ge \RB
\]
 and hence asserts that $X(S)$ is not Ramanujan.
\end{proof}

\begin{proof}
[Proof of Theorem~\ref{thm:finiteabel}]
 From the fundamental theorem of finite abelian groups,
 we may assume that $G$ is of the form $G=\mathbb{Z}_{m_1}\oplus\mathbb{Z}_{m_2}\oplus \cdots \oplus \mathbb{Z}_{m_r}$ 
 where $m_1,m_2,\ldots,m_r$ are odd integers with $m_1\,|\,m_2\,|\,\cdots \,|\,m_r$.
 Moreover, because $G$ is not cyclic, we may further assume that $r\ge 2$.
 Let $G_1=\mathbb{Z}_{m_1}$ and $G_2=\mathbb{Z}_{m_2}\oplus \cdots \oplus \mathbb{Z}_{m_r}$ and
 $n_1=m_1$ and $n_2=m_2\cdots m_r$, respectively.
 If $r\ge 3$, then one easily sees that $n_2\ge 4n_1-3$ and hence,
 from Lemma~\ref{lem:keyabel}, $G$ is ordinary.
 Therefore, it is sufficient to study only the case $r=2$.

 Suppose that at least one of $m_1$ and $m_2$ has two prime factors.
 Then, since $m_1\,|\,m_2$, it can be written as $m_1=p^{e_1}t_1$ and $m_2=p^{e_2}t_2$
 for some odd prime $p$ and odd integers $t_1,t_2$ with $(p,t_1)=(p,t_2)=1$.
 Here, at least one of $t_1$ and $t_2$ are greater than one.
 Now, from the Chinese reminder theorem, we have  
 $G=\mathbb{Z}_{m_1}\oplus \mathbb{Z}_{m_2}\cong \mathbb{Z}_{p^{e_1}}\oplus \mathbb{Z}_{t_1}\oplus \mathbb{Z}_{p^{e_2}}\oplus \mathbb{Z}_{t_2}$,
 whence, from Lemma~\ref{lem:keyabel} again, $G$ is ordinary.
 Therefore, we may assume that $m_1$ and $m_2$ can be respectively written as $m_1=p^s$ and $m_2=p^{t}$ for some $p$ and $s\le t$.  
 Moreover, we see that $n_2\ge 4n_1-3$ if $t\ge s+2$ or $t=s+1$ with $s=1$ when $p=3$ or $t\ge s+1$ when $p\ge 5$.
 Hence, it is enough to consider only the cases $(m_1,m_2)=(3^s,3^{s+1})$ with $s\ge 2$ or $(m_1,m_2)=(p^s,p^s)$ with
 $s\ge 1$ for $p\ge 5$.

 Assume that $G$ is the former, that is, $G=\mathbb{Z}_{3^s}\oplus\mathbb{Z}_{3^{s+1}}$ with $s\ge 2$.
 In this case, we can take $S\in\cS_{l_0+2}$ as $G\setminus S\subset \{(g_1,g_2)\in G\,|\,3\,|\,g_1, \ 1\le g_2\le 3^{s+1}\}$
 because $l_0+2=2\Gauss{3^s\sqrt{3}-\frac{3}{2}}+3<3^{2s}$.
 Then, for such $S$, we have $|\lambda_{\chi_0}|=l_0+2$ where $\chi_0(g_1,g_2)=e^{\frac{2\pi ig_1}{3}}$.
 This shows that $G$ is ordinary.
 We next consider the latter, that is, $G=\mathbb{Z}_{p^s}\oplus\mathbb{Z}_{p^{s}}$ with $s\ge 1$.
 At first, let $s\ge 2$.
 Then, we can similarly take $S\in\cS_{l_0+2}$ as $G\setminus S\subset \{(g_1,g_2)\in G\,|\,p\,|\,g_1, \ 1\le g_2\le p^{s}\}$
 because $l_0+2=2p^s-1<p^{2s-1}$ and hence, by the same reason as above, 
% have $|\lambda_{\chi_1}|=l_0+2$ where $\chi_1(g_1,g_2)=e^{\frac{2\pi ig_1}{p}}$.
% This shows that
 $G$ is ordinary.

 Now, only the cases $G=\mathbb{Z}_{p}\oplus\mathbb{Z}_{p}$ with $p\ge 5$ are left. 
% When $p=3$, it is easy to check that all Cayley graphs of $G$ are Ramanujan,
% which in particular implies that $\mathbb{Z}_{3}\oplus\mathbb{Z}_{3}$ is exceptional.
% Hence, from now on, we assume that $p\ge 5$.
 Let $h\ge 1$.
 If $p\ge 2h-3$, then, we can take $S\in\cS_{l_0+2h}$ as
 $G\setminus S\subset\{(0,g_2)\in G\,|\,1\le g_2\le p\}\cup\{(\pm 1,\pm g_2)\in G\,|\,1\le g_2\le \frac{p-1}{2}\}\cup\{(\pm 1,0)\}$ because
 $l_0+2h(=2p-3+2h)\le 3p$.
 For such $S$, we have $|\lambda_{\chi_0}|=p+(p-3+2h)\cos\frac{2\pi}{p}$ where $\chi_0(g_1,g_2)=e^{\frac{2\pi ig_1}{p}}$.
 We notice that this is the largest, that is, $|\lambda_{\chi_0}|=\hat{\mu}_{l_0+2h}=\underset{S\in\mathcal{S}_{l_0+2h}}{\max}\mu(S)$.
 Since $\hat{l}\ge l_0+2h$ is equivalent to $\hat{\mu}_{l_0+2h}\le \RB_{l_0+2h}$,
 this implies that $p+(p-3+2h)\cos\frac{2\pi}{p}\le 2\sqrt{p^2-2p+2-2h}$. Let 
\[
 d(p,h)=p+(p-3+2h)\cos\frac{2\pi}{p}-2\sqrt{p^2-2p+2-2h}.
\]
 Then, one can see that $d(p,1)>0$ if and only if $p\ge 19$, $d(p,2)>0$ if and only if $p\ge 7$ and $d(p,3)>0$ for all $p\ge 5$.
 This completes the proof. 
\end{proof}

%===============   Reference  ===============================================

\bigskip 

\noindent
\textsc{Miki HIRANO}\\
 Graduate School of Science and Engineering, Ehime University,\\
 Bunkyo-cho, Matsuyama, 790-8577 JAPAN.\\
 \texttt{hirano@math.sci.ehime-u.ac.jp}\\

\noindent
\textsc{Kohei KATATA}\\
 Graduate School of Science and Engineering, Ehime University,\\
 Bunkyo-cho, Matsuyama, 790-8577 JAPAN.\\
 \texttt{katata@math.sci.ehime-u.ac.jp}\\

\noindent
\textsc{Yoshinori YAMASAKI}\\
 Graduate School of Science and Engineering, Ehime University,\\
 Bunkyo-cho, Matsuyama, 790-8577 JAPAN.\\
 \texttt{yamasaki@math.sci.ehime-u.ac.jp}

%===========================================================================
%===========================================================================

\begin{thebibliography}{99999}
%%%%%%%%%%%%%%%%%%%%%%%%%%%%%%%%%%%%%%%%%%%%%%%%%%%%%%%%%%%%%%%%%%%%%%%%%%%%%
\bibitem[AR]{AlonRoichman1994}
 N. Alon and Y. Roichman,
 Random Cayley graphs and expanders,
 {\it Random Structures Algorithms}, {\bf 5} (1994), 271--284.
%%%%%%%%%%%%%%%%%%%%%%%%%%%%%%%%%%%%%%%%%%%%%%%%%%%%%%%%%%%%%%%%%%%%%%%%%%%%%
\bibitem[BH]{BatemanHorn1962}
 P. T. Bateman and R. A. Horn,
 A heuristic asymptotic formula concerning the distribution of prime numbers,
 {\it Math. Comput.} {\bf 16} (1962), 363--367.
%%%%%%%%%%%%%%%%%%%%%%%%%%%%%%%%%%%%%%%%%%%%%%%%%%%%%%%%%%%%%%%%%%%%%%%%%%%%%
%\bibitem[Ch]{Chiu1992}
% P. Chiu, 
% Cubic Ramanujan graphs, 
% {\it Combinatorica}, {\bf 12} (1992), 275--285. 
%%%%%%%%%%%%%%%%%%%%%%%%%%%%%%%%%%%%%%%%%%%%%%%%%%%%%%%%%%%%%%%%%%%%%%%%%%%%%
%\bibitem[Co]{Conrad2003}
% K. Conrad,
% Hardy-Littlewood constants,
% Mathematical properties of sequences and other combinatorial structures
% (Los Angeles, CA, 2002), 133--154, Kluwer Acad. Publ., Boston, MA, 2003. 
%%%%%%%%%%%%%%%%%%%%%%%%%%%%%%%%%%%%%%%%%%%%%%%%%%%%%%%%%%%%%%%%%%%%%%%%%%%%%
\bibitem[DSV]{DavidoffSarnakValette2003}
 G. Davidoff, P. Sarnak and A. Valette,
 Elementary number theory, group theory, and Ramanujan graphs, 
 London Mathematical Society Student Texts, 55.
 Cambridge University Press, Cambridge, 2003.
%%%%%%%%%%%%%%%%%%%%%%%%%%%%%%%%%%%%%%%%%%%%%%%%%%%%%%%%%%%%%%%%%%%%%%%%%%%%%
\bibitem[DM]{DeckerMoree2008}
 A. Decker and P. Moree, 
 Counting RSA-integers,
 {\it Results Math.}, {\bf 52} (2008), 35--39. 
%%%%%%%%%%%%%%%%%%%%%%%%%%%%%%%%%%%%%%%%%%%%%%%%%%%%%%%%%%%%%%%%%%%%%%%%%%%%%
\bibitem[DW]{Dieter1975}
 Dieter Wolke,
 Zur Gleichverteilung einiger Zahlenfolgen,
 {\it Math. Z.}, {\bf 142} (1975), 181--184.
%%%%%%%%%%%%%%%%%%%%%%%%%%%%%%%%%%%%%%%%%%%%%%%%%%%%%%%%%%%%%%%%%%%%%%%%%%%%%
\bibitem[DL]{Dieter1976}
 Dieter Leitmann,
 On the uniform distribution of some sequences,
 {\it J. London Math. Soc.}, {\bf 14} (1976), 430--432. 
%%%%%%%%%%%%%%%%%%%%%%%%%%%%%%%%%%%%%%%%%%%%%%%%%%%%%%%%%%%%%%%%%%%%%%%%%%%%%
\bibitem[HL]{HardyLittlewood1923}
 G. H. Hardy and J. E. Littlewood,
 Some problems of 'partitio numerorum'; I\hspace{-.1em}I\hspace{-.1em}I:
 on the expression of a number as a sum of primes,
 {\it Acta Math.} {\bf 44} (1923), 1--70.
%%%%%%%%%%%%%%%%%%%%%%%%%%%%%%%%%%%%%%%%%%%%%%%%%%%%%%%%%%%%%%%%%%%%%%%%%%%%%
\bibitem[HW]{HardyWright1979}
 G. H. Hardy and E. M. Wright,
 An introduction to the theory of numbers. Fifth edition.
 The Clarendon Press, Oxford University Press, New York, 1979.
%%%%%%%%%%%%%%%%%%%%%%%%%%%%%%%%%%%%%%%%%%%%%%%%%%%%%%%%%%%%%%%%%%%%%%%%%%%%%
\bibitem[Ha]{Hashimoto2009}
 Y. Hashimoto,
 On asymptotic behavior of composite integers $n=pq$,
 {\it J. Math-for-Ind.}, {\bf 1}  (2009), 45--49.
%%%%%%%%%%%%%%%%%%%%%%%%%%%%%%%%%%%%%%%%%%%%%%%%%%%%%%%%%%%%%%%%%%%%%%%%%%%%%
\bibitem[He]{Heuberger2003}
 C. Heuberger,
 On planarity and colorability of circulant graphs,
 {\it Discrete Math.}, {\bf 268} (2003), 153--169. 
%%%%%%%%%%%%%%%%%%%%%%%%%%%%%%%%%%%%%%%%%%%%%%%%%%%%%%%%%%%%%%%%%%%%%%%%%%%%%
\bibitem[I]{Iwaniec1978}
 H. Iwaniec, 
 Almost-primes represented by quadratic polynomials,
 {\it Invent. Math.}, {\bf 47} (1978), 171--188. 
%%%%%%%%%%%%%%%%%%%%%%%%%%%%%%%%%%%%%%%%%%%%%%%%%%%%%%%%%%%%%%%%%%%%%%%%%%%%%
\bibitem[K]{Katata2013}
 K. Katata, 
 On Ramanujan circulant graphs of even order,
 preprint, 2015. 
%%%%%%%%%%%%%%%%%%%%%%%%%%%%%%%%%%%%%%%%%%%%%%%%%%%%%%%%%%%%%%%%%%%%%%%%%%%%%
\bibitem[HKY]{HiranoKatataYamasaki}
 M. Hirano, K. Katata and Y. Yamasaki, 
 Ramanujan Cayley graphs of Frobenius groups,
 submitted, 2015. 
%%%%%%%%%%%%%%%%%%%%%%%%%%%%%%%%%%%%%%%%%%%%%%%%%%%%%%%%%%%%%%%%%%%%%%%%%%%%%
\bibitem[HLW]{HooryLinialWigderson2006}
 S. Hoory, N. Linial and A. Wigderson, 
 Expander graphs and their applications,
 {\it Bull. Amer. Math. Soc.}, {\bf 43} (2006), 439--561.
%%%%%%%%%%%%%%%%%%%%%%%%%%%%%%%%%%%%%%%%%%%%%%%%%%%%%%%%%%%%%%%%%%%%%%%%%%%%%
\bibitem[La]{Landau1900}
 E. Landau, 
 Sur quelques problemes relatifs a la distribution des nombres premiers,
 {\it  Bull. Soc. Math. France}, {\bf 28} (1900), 25--38.
%%%%%%%%%%%%%%%%%%%%%%%%%%%%%%%%%%%%%%%%%%%%%%%%%%%%%%%%%%%%%%%%%%%%%%%%%%%%%
\bibitem[Le]{LemkeOliver2012}
 R. Lemke Oliver, 
 Almost-primes represented by quadratic polynomials.
 {\it Acta Arith.}, {\bf 151} (2012), 241--261.
%%%%%%%%%%%%%%%%%%%%%%%%%%%%%%%%%%%%%%%%%%%%%%%%%%%%%%%%%%%%%%%%%%%%%%%%%%%%%
%\bibitem[LNS]{LeungNguyenSo2011}
% Ka Hin Leung, Vinh Nguyen and Wasin So,
% Nonexistence of a circulant expander family,
% {\it Bull. Aust. Math. Soc.}, {\bf 83} (2011), 87--95. 
%%%%%%%%%%%%%%%%%%%%%%%%%%%%%%%%%%%%%%%%%%%%%%%%%%%%%%%%%%%%%%%%%%%%%%%%%%%%%
%\bibitem[L]{Loo2011}
% A. Loo,
% On the primes in the interval $[3n,4n]$,
% {\it Int. J. Contemp. Math. Sci.}, {\bf 6} (2011), 1871--1882. 
%%%%%%%%%%%%%%%%%%%%%%%%%%%%%%%%%%%%%%%%%%%%%%%%%%%%%%%%%%%%%%%%%%%%%%%%%%%%%
%\bibitem[Lu1]{Lubotzky1994}
% A. Lubotzky,
% Discrete groups, expanding graphs and invariant measures.
% With an appendix by Jonathan D. Rogawski.
% Progress in Mathematics, 125. Birkhauser Verlag, Basel, 1994.
%%%%%%%%%%%%%%%%%%%%%%%%%%%%%%%%%%%%%%%%%%%%%%%%%%%%%%%%%%%%%%%%%%%%%%%%%%%%
\bibitem[Lu]{Lubotzky2012}
 A. Lubotzky,
 Expander graphs in pure and applied mathematics,
 {\it Bull. Amer. Math. Soc.}, {\bf 49} (2012), 113--162.  
%%%%%%%%%%%%%%%%%%%%%%%%%%%%%%%%%%%%%%%%%%%%%%%%%%%%%%%%%%%%%%%%%%%%%%%%%%%%
\bibitem[LPS]{LubotzkyPhillipsSarnak1988}
 A. Lubotzky, R. Phillips and P. Sarnak,
 Ramanujan graphs,
 {\it Combinatorica}, {\bf 8} (1988), 261--277.  
%%%%%%%%%%%%%%%%%%%%%%%%%%%%%%%%%%%%%%%%%%%%%%%%%%%%%%%%%%%%%%%%%%%%%%%%%%%%%
%\bibitem[Ma]{Margulis1988}
% G. A. Margulis,
% Explicit group-theoretic constructions of combinatorial schemes and their applications 
% in the construction of expanders and concentrators,
% {\it Problemy Peredachi Informatsii}, {\bf 24} (1988), 51--60.
%%%%%%%%%%%%%%%%%%%%%%%%%%%%%%%%%%%%%%%%%%%%%%%%%%%%%%%%%%%%%%%%%%%%%%%%%%%%%
\bibitem[MV]{MontgomeryVaughan2007}
 H. Montgomery and R. Vaughan, 
 Multiplicative number theory. I. Classical theory,
 Cambridge Studies in Advanced Mathematics, 97,
 Cambridge University Press, Cambridge, 2007. 
%%%%%%%%%%%%%%%%%%%%%%%%%%%%%%%%%%%%%%%%%%%%%%%%%%%%%%%%%%%%%%%%%%%%%%%%%%%%%
%\bibitem[Mo]{Morgenstern1994}
% M. Morgenstern,
% Existence and explicit constructions of $q+1$  regular Ramanujan graphs for every prime power $q$,
% {\it J. Combin. Theory Ser. B}, {\bf 62} (1994), 44--62.
%%%%%%%%%%%%%%%%%%%%%%%%%%%%%%%%%%%%%%%%%%%%%%%%%%%%%%%%%%%%%%%%%%%%%%%%%%%%%
%\bibitem[Mu]{Murty2003}
% M. R. Murty,
%  Ramanujan graphs,
% {\it J. Ramanujan Math. Soc.}, {\bf 18} (2003), 33--52. 
%%%%%%%%%%%%%%%%%%%%%%%%%%%%%%%%%%%%%%%%%%%%%%%%%%%%%%%%%%%%%%%%%%%%%%%%%%%%%
\bibitem[T1]{Terras1999}
 A. Terras, 
 Fourier analysis on finite groups and applications,
 London Mathematical Society Student Texts, {\bf 43},
 Cambridge University Press, Cambridge, 1999.
%%%%%%%%%%%%%%%%%%%%%%%%%%%%%%%%%%%%%%%%%%%%%%%%%%%%%%%%%%%%%%%%%%%%%%%%%%%%%
\bibitem[T2]{Terras2011}
 A. Terras, 
 Zeta functions of graphs, A stroll through the garden.
 Cambridge Studies in Advanced Mathematics, 128.
 Cambridge University Press, Cambridge, 2011.
%%%%%%%%%%%%%%%%%%%%%%%%%%%%%%%%%%%%%%%%%%%%%%%%%%%%%%%%%%%%%%%%%%%%%%%%%%%%%
\end{thebibliography}
\end{document}